\documentclass[11pt]{article}
\usepackage[english]{babel}
\usepackage{amssymb,amsmath,amsthm}
\newtheorem{theorem}{Theorem}[section]
\newtheorem{lemma}[theorem]{Lemma}
\newtheorem{proposition}[theorem]{Proposition}
\newtheorem{corollary}[theorem]{Corollary}

\newtheorem{question}[theorem]{Question}

{\theoremstyle{definition}}
{\theoremstyle{definition}\newtheorem{example}[theorem]{Example}}
{\theoremstyle{definition}}
{\theoremstyle{definition}}

\newtheorem*{thmbp}{Theorem BP}

\newtheorem*{thmw}{Theorem W}
\newtheorem*{thmlm}{Proposition LM}

\numberwithin{equation}{section}

\def\C{{\mathbb C}}
\def\PP{{\mathbb P}}
\def\N{{\mathbb N}}
\def\Z{{\mathbb Z}}
\def\R{{\mathbb R}}
\def\Q{{\mathbb Q}}
\def\T{{\mathbb T}}
\def\K{{\mathbb K}}

\def\pp{{\mathcal P}}
\def\epsilon{\varepsilon}
\def\kappa{\varkappa}
\def\phi{\varphi}
\def\leq{\leqslant}
\def\geq{\geqslant}

\def\dim{\hbox{\tt dim}\,}
\def\ker{\hbox{\tt ker}\,}

\def\spann{\hbox{\tt span}\,}

\def\bin#1#2{\left({{#1}\atop {#2}}\right)}

\def\uu{{\mathcal U}}

\title{Universal elements for non-linear operators and their
applications}

\author{Stanislav Shkarin}

\date{}

\begin{document}

\maketitle

\begin{abstract} We prove that under certain topological conditions
on the set of universal elements of a continuous map $T$ acting on a
topological space $X$, that the direct sum $T\oplus M_g$ is
universal, where $M_g$ is multiplication by a generating element of
a compact topological group. We use this result to characterize
$\R_+$-supercyclic operators and to show that whenever $T$ is a
supercyclic operator and $z_1,\dots,z_n$ are pairwise different
non-zero complex numbers, then the operator $z_1T\oplus
{\dots}\oplus z_n T$ is cyclic. The latter answers affirmatively a
question of Bayart and Matheron.
\end{abstract}



\section{Introduction \label{s1}}\rm

All topological spaces in this article {\bf are assumed to be
Hausdorff} and all vector spaces are supposed to be over the field
$\K$ being either the field $\C$ of complex numbers or the field
$\R$ of real numbers. As usual, $\R_+$ is the set of non-negative
real numbers, $\Q$ is the field of rational numbers, $\Z$ is the set
of integers, $\Z_+$ is the set of non-negative integers, $\N$ is the
set of positive integers and $\T=\{z\in\C:|z|=1\}$. Symbol $L(X)$
stands for the space of continuous linear operators on a topological
vector space $X$ and $X^*$ is the space of continuous linear
functionals on $X$. For each $T\in L(X)$, the dual operator
$T^*:X^*\to X^*$ is defined as usual: $(T^*f)(x)=f(Tx)$ for $f\in
X^*$ and $x\in X$. It is worth noting that if $X$ is not locally
convex, then the elements of $X^*$ may not separate points of $X$. A
family ${\cal F}=\{F_a:a\in A\}$ of continuous maps from a
topological space $X$ to a topological space $Y$ is called {\it
universal} if there is $x\in X$ for which the orbit $O({\cal
F},x)=\{F_ax:a\in A\}$ is dense in $Y$. Such an $x$ is called a {\it
universal element} for $\cal F$. We use the symbol $\uu({\cal F})$
for the set of universal elements for $\cal F$. If $X$ is a
topological space and $T:X\to X$ is a continuous map, then we say
that $x\in X$ is {\it universal} for $T$ if $x$ is universal for the
family $\{T^n:n\in\Z_+\}$. That is, $x$ is universal for $T$ if the
orbit $O(T,x)=\{T^nx:n\in\Z_+\}$ is dense in $X$. We denote the set
of universal elements for $T$ by $\uu(T)$. That is,
$\uu(T)=\uu(\{T^n:n\in\Z_+\})$. In order to formulate the following
theorem, we need to recall few topological definitions. A
topological space $X$ is called {\it connected} if it has no subsets
different from $\varnothing$ and $X$, which are closed and open. A
topological space $X$ is called {\it path connected} if for each
$x,y\in X$, there is a continuous map $f:[0,1]\to X$ such that
$f(0)=x$ and $f(1)=y$. A topological space $X$ is called {\it simply
connected} if for any continuous function $f:\T\to X$, there exist a
continuous function $F:\T\times [0,1]\to X$ and $x_0\in X$ such that
$F(z,0)=f(z)$ and $F(z,1)=x_0$ for any $z\in\T$. Next, $X$ is called
{\it locally path connected at} $x\in X$  if for any neighborhood
$U$ of $x$, there exists a neighborhood $V$ of $x$ such that for any
$y\in V$, there is a continuous map $f:[0,1]\to X$ such that
$f(0)=x$, $f(1)=y$ and $f([0,1])\subseteq U$. A space $X$ is called
{\it locally path connected} if it is locally path connected at
every point. Equivalently, $X$ is locally path connected if there is
a base of topology of $X$ consisting of path connected sets.
Finally, we recall that a topological space $X$ is called {\it
Baire} if for any sequence $\{U_n\}_{n\in\Z_+}$ of dense open
subsets of $X$, the intersection of $U_n$ is dense in $X$. We say
that an element $g$ of a topological group $G$ is its {\it
generator} if $\{g^n:n\in\Z_+\}$ is dense in $G$. It is worth
mentioning that each topological group, which has a  generator, is
abelian. If $T:X\to X$ and $S:Y\to Y$ are two maps, we use the
symbol $T\oplus S$ to denote the map
$$
T\oplus S:X\times Y\to X\times Y,\qquad (T\oplus S)(x,y)=(Tx,Sy).
$$
The following theorem is a generalization of the Ansari theorem on
hypercyclicity of the powers of a hypercyclic operator in several
directions.

\begin{theorem} \label{general} Let $X$ be a topological space, $T:X\to X$ be a
continuous map and $g$ be a generator of a compact topological group
$G$. Assume also that there is a non-empty subset $Y$ of $\uu(T)$
such that $T(Y)\subseteq Y$ and $Y$ is path connected, locally path
connected and simply connected. Then the set
$\{(T^nx,g^n):n\in\Z_+\}$ is dense in $X\times G$ for any $x\in Y$.
\end{theorem}

In the case when $X$ is compact and metrizable, the above theorem
follows from Theorem~11.2 in the paper \cite{distal} by Furstenburg.
Our proof is based on the same general idea as in \cite{distal},
which is also reproduced in the proofs of main results in
\cite{semi,muller}. We would like to mention the following immediate
corollary of Theorem~\ref{general}.

\begin{corollary} \label{general0} Let $X$ be a topological space, $T:X\to X$ be a
continuous map and $g$ be a generator of a compact topological group
$G$. Assume also that there is a non-empty subset $Y$ of $\uu(T)$
such that $T(Y)\subseteq Y$ and $Y$ is path connected, locally path
connected and simply connected. Then $Y\times G\subseteq \uu(T\oplus
M_g)$, where $M_g:G\to G$, $M_gh=gh$.
\end{corollary}

If $X$ is a topological vector space and $T\in L(X)$, then a
universal element for $T$ is called {\it a hypercyclic vector} for
$T$ and the operator $T$ is called {\it hypercyclic} if
$\uu(T)\neq\varnothing$. Moreover, $x\in \uu(\{sT^n:s\in\K,\
n\in\Z_+\})$ is called a {\it supercyclic vector for} $T$ and $T$ is
called {\it supercyclic} if it has supercyclic vectors. Similarly
$T$ is called $\R_+$-{\it supercyclic} if the family ${\cal
F}=\{sT^n:s\in\R_+,\ n\in\Z_+\}$ is universal and elements of
$\uu({\cal F})$ are called $\R_+$-{\it supercyclic vectors} for $T$.
We refer to surveys \cite{ge1,ge2,msa} for additional information on
hypercyclicity and supercyclicity. The question of characterizing of
$\R_+$-supercyclic operators on complex Banach spaces was raised in
\cite{bermud,muller}. Maria de la Rosa has recently demonstrated
that the answer conjectured in \cite{bermud,muller} is indeed true.
We obtain the same result in the more general setting of topological
vector spaces by means of applying Theorem~\ref{general}.

\begin{theorem} \label{rp} A continuous linear operator $T$ on a
complex infinite dimensional topological vector space $X$ is
$\R_+$-supercyclic if and only if $T$ is supercyclic and either the
point spectrum $\sigma_p(T^*)$ of the dual operator $T$ is empty or
$\sigma_p(T^*)=\{z\}$, where $z\in\C\setminus\{0\}$ and $z/|z|$ has
infinite order in the group $\T$.
\end{theorem}

Recall also that a vector $x$ from a topological vector space $X$ is
called a {\it cyclic} vector for $T\in L(X)$ if
$\spann\{T^nx:n\in\Z_+\}$ is dense in $X$ and $T$ is called cyclic
if it has cyclic vectors. In \cite{gri} it is observed that if $T$
is a hypercyclic operator on a Banach space, then $T\oplus (-T)$ is
cyclic. It is shown in \cite{bm3} that if $T$ is a supercyclic
operator on a complex Banach space and $z_1,\dots,z_n$ are pairwise
different complex numbers such that $z_1^k={\dots}=z_n^k=1$ for some
$k\in\N$, then $z_1T\oplus{\dots}\oplus z_nT$ is cyclic. It is also
asked in \cite{bm3} whether the condition $z_1^k={\dots}=z_n^k=1$
can be removed. The next theorem provides an affirmative answer to
this question.

\begin{theorem}\label{zo} Let $n\in\N$ and $T$ be a supercyclic
continuous linear operator on an infinite dimensional topological
vector space $X$. Then for any pairwise different non-zero
$z_1,\dots,z_n\in\K$, the operator $z_1T\oplus{\dots}\oplus z_nT$ is
cyclic.
\end{theorem}

Theorem~\ref{general} is proved in Section~\ref{proof}. In
Section~\ref{connect} we formulate few connectedness related lammas,
needed for application of Theorem~\ref{general}. The proof of these
lemmas is postponed until the last section. Section~\ref{applic1} is
devoted to some straightforward applications of
Theorem~\ref{general}. In particular, it is shown that Ansari's
theorems \cite{ansa} on hypercyclicity of powers of a hypercyclic
operator and supercyclicity of powers of supercyclic operators, the
Le\'on-Saavedra and M\"uller theorem \cite{muller} on hypercyclicity
of rotations of hypercyclic operators and the M\"uller and Peris
\cite{semi} theorem on hypercyclicity of each operator $T_t$ with
$t>0$ in a strongly continuous universal semigroup $\{T_t\}_{t\geq
0}$ of continuous linear operators acting on a complete metrizable
topological vector space $X$, all follow from Theorem~\ref{general}.
In Section~\ref{rsup}, Theorem~\ref{rp}, characterizing
$\R_+$-supercyclic operators, is proven. Theorems~\ref{zo} is proved
in Section~\ref{dirsum}. In Section~\ref{remarks} we discuss the
structure of supercyclic operators $T$ with non-empty
$\sigma_p(T^*)$, prove cyclicity of finite direct sums of operators
satisfying the Supercyclicity Criterion with the same sequence
$\{n_k\}$ and raise few questions.

\section{Proof of Theorem~\ref{general} \label{proof}}

\begin{lemma}\label{ll1} Let $X$ be a topological space with no isolated points
and $T:X\to X$ be a continuous map. Then $T(\uu(T))\subseteq
\uu(T)$. Moreover, each $x\in\uu(T)$ belongs to
$\overline{O(T,Tx)}\setminus O(T,Tx)$.
\end{lemma}

\begin{proof}Let $x\in \uu(T)$. Then the orbit $O(T,x)=\{T^kx:k\in\Z_+\}$
is dense in $X$. Since $X$ has no isolated points,
$O(T,Tx)=O(T,x)\setminus \{x\}$ is also dense in $X$. Hence $Tx\in
\uu(T)$. Thus $T(\uu(T))\subseteq \uu(T)$. Since $Tx\in\uu(T)$, we
have $x\in\overline{O(T,Tx)}$. It remains to show that $x\notin
O(T,Tx)$. Assume that $x\in O(T,Tx)$. Then $x=T^nx$ for some
$n\in\N$ and $O(T,x)$ is finite: $O(T,x)=\{T^kx:0\leq k<n\}$. Since
any finite set is closed in $X$, $O(T,x)=\overline{O(T,x)}=X$. Hence
$X$ is finite and therefore does have isolated points. This
contradiction completes the proof.
\end{proof}

The following two lemmas are well-known and could be found in any
textbook treating topological groups, see, for instance, \cite{HR}.

\begin{lemma} \label{sub} A closed subsemigroup of a compact
topological group is a subgroup.
\end{lemma}

\begin{lemma} \label{char} Let $G$ be a compact abelian topological
group, $g\in G$ and $g\neq 1_G$. Then there exists a continuous
homomorphism $\phi:G\to\T$ such that $\phi(g)\neq 1$.
\end{lemma}

We would also like to remind the following topological fact, see for
instance, \cite{angel}.

\begin{lemma} \label{clos} Let $X$ be a topological space, $Y$ be a
compact topological space and $\pi:X\times Y\to X$ be the
projection: $\pi(x,y)=x$. Then the map $\pi$ is closed.
\end{lemma}

We also need the following lemma, which borrows heavily from the
constructions in \cite{muller,semi}.

\begin{lemma} \label{L2} Let $X$ be a topological space, $\Lambda$
be a subsemigroup of the semigroup $C(X)$ of continuous maps from
$X$ to $X$ and $\uu=\uu(\Lambda)$ be the set of universal elements
for the family $\Lambda$. Let also $G$ be a compact abelian
topological group and $\phi:\Lambda\to G$ be a homomorphism:
$\phi(TS)=\phi(T)\phi(S)$ for any $T,S\in\Lambda$. For each $x,y\in
X$ we denote
\begin{equation}\label{nx0}
N_x=\overline{\{(Tx,\phi(T)):T\in\Lambda\}}\ \ \text{and}\ \
F_{x,y}=\{h\in G:(y,h)\in N_x\}.
\end{equation}
Then
\begin{itemize}
\item[\rm (a1)]$F_{x,y}$ is closed in $G$ for any $x,y\in X$ and
$F_{x,y}\neq \varnothing$ if  $x\in\uu$,
\item[\rm (a2)]$F_{x,y}F_{y,u}\subseteq F_{x,u}$ for any $x,y,u\in X$,
\item[\rm (a3)]$F_{x,x}=H$ is a closed subgroup of $G$ for any
$x\in\uu$, which does not depend on the choice of $x\in\uu$ and for
any $x,y\in\uu$, $F_{x,y}$ is a coset of $H$,
\item[\rm (a4)]the function $f:\uu\times\uu\to G/H$, $f(x,y)=F_{x,y}$
is separately continuous and satisfies $f(x,y)=f(y,x)^{-1}$,
$f(x,y)f(y,u)=f(x,u)$, and $f(x,x)=1_{G/H}$ for any $x,y,u\in\uu$.
Moreover, if $T\in\Lambda$ and $x,y,Ty\in \uu$, then
$f(x,Ty)=\phi(T)f(x,y)$.
\end{itemize}
\end{lemma}

\begin{proof} (a1) Closeness of $F_{x,y}$ follows from the closeness of $N_x$ and
the obvious equality  $\{y\}\times F_{x,y}=N_x\cap (\{y\}\times G)$.
Assume now that $x\in\uu$, $y\in X$ and $F_{x,y}=\varnothing$. Then
for any $h\in G$, $(y,h)$ does not belong to the closed set
$N_x\subset X\times G$. Hence we can pick open neighborhoods $U_h$
and $V_h$ of $y$ and $h$ in $X$ and $G$ respectively such that
$(U_h\times V_h)\cap N_x=\varnothing$. Since $G$ is compact, the
open cover $\{V_h:h\in G\}$ has a finite subcover
$\{V_{h_1},\dots,V_{h_n}\}$. Then
$$
(G\times U)\cap N_x\subseteq \bigcup_{j=1}^n(U_{h_j}\times
V_{h_j})\cap N_x=\varnothing,\ \ \text{where}\ \ U=\bigcap_{j=1}^n
U_{h_j}.
$$
From the definition of $N_x$ it follows that the set
$\{Tx:T\in\Lambda\}$ does not intersect $U$, which is open in $X$
and non-empty since $y\in U$. This contradicts universality of $x$
for $\Lambda$.

(a2) Let $a\in F_{x,y}$ and $b\in F_{y,u}$. We have to demonstrate
that $ab\in F_{x,u}$. Let $U$ be any neighborhood of $u$ in $X$ and
$V$ be any neighborhood of $1_G$ in $G$. Pick a neighborhood $W$ of
$1_G$ in $G$ such that $W\cdot W\subseteq V$. Since $b\in F_{y,u}$,
there exists $T\in\Lambda$ such that $Ty\in U$ and $b^{-1}\phi(T)\in
W$. Since $T$ is continuous, $T^{-1}(U)$ is a neighborhood of $y$.
Since $a\in F_{x,y}$, there exists $S\in \Lambda$ such that $Sx\in
T^{-1}(U)$ and $a^{-1}\phi(S)\in W$. Then $TSx\in T(T^{-1}(U))=U$
and $(ab)^{-1}\phi(TS)=a^{-1}\phi(S)b^{-1}\phi(T)\in W\cdot
W\subseteq V$ (here we use commutativity of $G$). That is,
$(TSx,\phi(TS))\in U\times abV$. Therefore the set
$\{(Rx,\phi(R)):R\in\Lambda\}$ intersects any neighborhood of
$(u,ab)$ in $X\times G$. Thus, $(u,ab)\in N_x$ and $ab\in F_{x,u}$.

(a3) Let $x\in\uu$. By (a1) and (a2) $F_{x,x}$ is closed nonempty
and satisfies $F_{x,x}F_{x,x}\subseteq F_{x,x}$. Thus, $F_{x,x}$ is
a closed subsemigroup of the compact topological group $G$. By
Lemma~\ref{sub}, $F_{x,x}$ is a closed subgroup of $G$. Let now
$x,y\in \uu$. According to (a2) $F_{x,y}F_{y,y}F_{y,x}\subseteq
F_{x,x}$. By (a1), we can pick $a\in F_{x,y}$ and $b\in F_{y,x}$. It
follows that $aF_{y,y}b=abF_{y,y}\subseteq F_{x,x}$. Since $F_{x,x}$
and $F_{y,y}$ are subgroups of $G$ and $abF_{y,y}\subseteq F_{x,x}$,
we have $F_{y,y}=abF_{y,y}(abF_{y,y})^{-1}\subseteq F_{x,x}$.
Similarly $F_{x,x}\subseteq F_{y,y}$. Hence $F_{x,x}=F_{y,y}$ and
therefore the subgroup $H=F_{x,x}$ does not depend on the choice of
$x\in\uu$. Now by (a2) $F_{x,y}H=F_{x,y}F_{y,y}\subseteq F_{x,y}$.
Thus $aH\subseteq F_{x,y}$. On the other hand, by (a2)
$F_{x,y}F_{y,x}\subseteq F_{x,x}=H$. Hence
$F_{x,y}b=bF_{x,y}\subseteq H$ and therefore $F_{x,y}\subseteq
b^{-1}H$. That is, $F_{x,y}$ is contained in a coset of $H$ and
contains a coset of $H$. It follows that $F_{x,y}$ is a coset of
$H$.

(a4) According to (a3) the function $f:\uu\times\uu\to G/H$,
$f(x,y)=F_{x,y}$ is well defined. Let $x,y,u\in \uu$. By (a3)
$F_{x,y}F_{y,u}\subseteq F_{x,u}$. According to (a4) $F_{x,y}$,
$F_{y,u}$ and $F_{x,u}$ are cosets of $H$. Hence
$F_{x,y}F_{y,u}=F_{x,u}$. It follows that $f(x,y)f(y,u)=f(x,u)$.
Similarly by (a3) and (a4) $F_{x,y}F_{y,x}=F_{x,x}=H$ and therefore
$f(x,y)f(y,x)=1_{G/H}$. Thus $f(y,x)=f(x,y)^{-1}$. Since
$F_{x,x}=H$, we have $f(x,x)=1_{G/H}$. Assume now that $T\in\Lambda$
and $Ty\in\uu$. From the definition of $F_{x,y}$, it immediately
follows that $\phi(T)\in F_{y,Ty}$. Hence
$f(x,Ty)=f(x,y)f(y,Ty)=\phi(T)f(x,y)$.

It remains to demonstrate that $f$ is separately continuous. Since
$f(y,x)=f(x,y)^{-1}$, it suffices to verify that for any fixed $x\in
\uu$, the function $\theta:\uu\to G/H$, $\theta(y)=f(x,y)$ is
continuous. Let $A$ be a closed subset of $G/H$ and
$A_0=\pi^{-1}(A)$, where $\pi:G\to G/H$ is the canonical projection.
Clearly $\theta^{-1}(A)=\{y\in\uu:F_{x,y}\subseteq A_0\}$ coincides
with $\pi_1(N_x\cap (\uu\times A_0))$, where $\pi_1:\uu\times G\to
\uu$ is the projection onto $\uu$: $\pi_1(v,h)=v$. Since $G$ is
compact, Lemma~\ref{clos} implies that the map $\pi_1$ is closed and
therefore $\theta^{-1}(A)=\pi_1(N_x\cap (\uu\times A_0))$ is closed
in $\uu$. Since $A$ is an arbitrary closed subset of $G/H$, $\theta$
is continuous.
\end{proof}

The following lemma is a particular case of Lemma~\ref{L2}.

\begin{lemma} \label{l2} Let $X$ be a topological space with no
isolated points, $T:X\to X$ be a continuous map, $g$ be an element
of a compact abelian topological group $G$ and $\uu=\uu(T)$. For
each $x,y\in X$ we denote
\begin{equation}\label{nx}
N_x=\overline{\{(T^nx,g^n):n\in\Z_+\}}\ \ \text{and}\ \
F_{x,y}=\{h\in G:(y,h)\in N_x\}.
\end{equation}
Then conditions {\rm (a1--a3)} of Lemma~$\ref{L2}$ are satisfied and
\begin{itemize}
\item[\rm (a$4'$)]the function $f:\uu\times\uu\to G/H$, $f(x,y)=F_{x,y}$
is separately continuous and satisfies $f(x,y)=f(y,x)^{-1}$,
$f(x,Ty)=gf(x,y)$, $f(x,y)f(y,u)=f(x,u)$, $f(x,x)=1_{G/H}$ for any
$x,y,u\in\uu$.
\end{itemize}
\end{lemma}

\begin{proof} We apply Lemma~\ref{L2} with
$\Lambda=\{T^n:n\in\Z_+\}$ and $\phi:\Lambda\to G$, $\phi(T^n)=g^n$.
Conditions (a1--a3) follow directly from Lemma~\ref{L2} and so does
(a4) since $T(\uu)\subseteq \uu$ according to Lemma~\ref{ll1} and
$\phi(T)=g$.
\end{proof}

Before proving Theorem~\ref{general}, we would like to introduce
some notation. Clearly, for any continuous function $f:[a,b]\to\T$,
there exists a unique, up to adding an integer, continuous function
$\phi:[a,b]\to \R$ such that $f(t)=e^{2\pi i\phi(t)}$. Then the
number $\phi(b)-\phi(a)\in \R$ is uniquely defined and called the
{\it winding number} of $f$ (sometimes also called the index of $f$
or the variation of the argument of $f$). We denote the winding
number of $f$ as $w(f)$. We would also like to remind the following
elementary and well-known properties of the winding number.

\begin{itemize}
\item[(w0)] If $f:[a,b]\to\T$ is continuous and $f(a)=f(b)$, then
$w(f)\in\Z$. Moreover, if also $g:[a,b]\to\T$ is continuous,
$g(a)=g(b)$ and there exists continuous $h:[a,b]\times [0,1]\to \T$
satisfying $h(t,0)=f(t)$, $h(t,1)=g(t)$ and $h(a,s)=h(b,s)$ for
$t\in[a,b]$ and $s\in[0,1]$, then $w(f)=w(g)$;
\item[(w1)] If $a<b<c$ and $f:[a,c]\to\T$ is continuous, then $w(f)=
w\bigl(f\bigr|_{[a,b]}\bigr)+w\bigl(f\bigr|_{[b,c]}\bigr)$;
\item[(w2)] If $f:[a,b]\to\T$, $h:[c,d]\to[a,b]$ are continuous, $h(c)=a$,
$h(d)=b$, then $w(f\circ h)= w(f)$;
\item[(w3)] If $f:[a,b]\to\T$ is continuous and $u\in\T$, then $w(uf)=
w(f)$;
\item[(w4)] If $f:[a,b]\to\T$ is continuous and not onto, then
$|w(f)|<1$.
\end{itemize}

The following lemma is a key ingredient of the proof of
Theorem~\ref{general}. It seems also that it has a potential for
different applications.

\begin{lemma}\label{map} Let $X$ be a path connected, locally path
connected and simply connected topological space, $T:X\to X$ be
continuous and and $g$ be an element of a compact abelian
topological group $G$. Assume also that there exists a continuous
map $f:X\to G$ such that $f(Ty)=gf(y)$ for any $y\in X$. Then either
$g=1_G$ or $O(T,x)$ is closed in $X$ for any $x\in X$.
\end{lemma}

\begin{proof} Assume that there exists $x\in X$ such that the orbit
$O(T,x)$ is non-closed and $g\neq 1_G$. By Lemma~\ref{char}, there
exists a continuous homomorphism $\phi:G\to \T$ such that
$z=\phi(g)\neq 1$. Since $\T$ has no closed subgroups except $\T$
and $U_n=\{u\in \T:u^n=1\}$ for $n\in\N$, we see that $\phi(G)$,
being a non-trivial closed subgroup of $\T$, coincides either with
$\T$ or with $U_n$ for some $n\geq 2$. Let $r=\phi\circ f$. Since
$\phi$ and $f$ are continuous, $r$ is a continuous map from $X$ to
$\T$. Moreover, from the properties of $\phi$ and $f$ it follows
that
\begin{equation}\label{gn1}
\text{$r(Tv)=zr(v)$ for any $v\in X$}.
\end{equation}

{\bf Case} \ $\phi(G)=U_n$ for some $n\geq 2$. Since $z\neq 1$,
(\ref{gn1}) implies that $r(Tv)\neq r(v)$ for any $v\in X$. Hence
$r(X)$ consists of more than one element. Thus, $r$ is a continuous
map from $X$ to $\T$, whose range, being a subset of $\phi(G)$, is
finite and consists of more than one element. The existence of such
a map contradicts connectedness of $X$.

{\bf Case} \ $\phi(G)=\T$. Pick $y\in \overline{O(T,x)}\setminus
O(T,x)$ and $z_0\in\T\setminus\{r(y)\}$. Since $r$ is continuous,
$r^{-1}(\T\setminus\{z_0\})$ is a neighborhood of $y$ in $X$. Since
$X$ is locally path connected, we can pick a neighborhood $U$ of $y$
such that $U$ is path connected and $U\subseteq
r^{-1}(\T\setminus\{z_0\})$. Since $y\in \overline{O(T,x)}$, we can
pick $n\in\Z_+$ such that $T^nx\in U$. Since $U$ is path connected,
there is continuous $\gamma_0:[0,1]\to U$ such that $\gamma_0(0)=y$,
$\gamma_0(1)=T^nx$. Since $X$ is path connected, there is continuous
$\alpha:[0,1]\to X$ such that $\alpha(0)=x$ and $\alpha(1)=Tx$. Let
$\beta:[0,1]\to\T$, $\beta=r\circ\alpha$. By (\ref{gn1}),
$\beta(1)/\beta(0)=z\neq1$. Hence the winding number $w(\beta)$ is
non-integer and therefore non-zero. Next, since $y\in
\overline{O(T,Tx)}\setminus O(T,x)$, we see that the set
$\{k\in\N:T^kx\in U\}$ is infinite. Thus we can pick $m\in\N$ such
that $T^{n+m}x\in U$ and $m|w(\beta_1)|>2$. Since $T^{n+m}x\in U$
and $U$ is path connected, there exists continuous
$\gamma_1:[0,1]\to U$ such that $\gamma_1(0)=T^{n+m}x$ and
$\gamma_1(1)=y$.

Consider the path $\rho:[0,m+2]\to X$ defined by the formula
\begin{equation}\label{path}
\rho(t)=\left\{\begin{array}{ll} \gamma_0(t)&\text{if}\ \ t\in[0,1);
\\ T^{n+k-1}\alpha(t-k)&\text{if}\ \ t\in[k,k+1),\ 1\leq k\leq m;\\
\gamma_1(t-m-1)&\text{if}\ \ t\in[m+1,m+2].
\end{array}
\right.
\end{equation}
Since $T^{n+k-1}\alpha(1)=T^{n+k}\alpha(0)=T^{n+k}x$ for $1\leq
k\leq m$, $\gamma_0(1)=T^n\alpha(0)=T^nx$,
$T^{n+m-1}\alpha(1)=\gamma_1(0)=T^{n+m}x$ and
$\gamma_0(0)=\gamma_1(1)=y$, we see that $\rho$ is continuous and
$\rho(0)=\rho(m+2)$. Since $X$ is simply connected, there exists a
continuous map $\tau:[0,m+2]\times [0,1]\to X$ such that
$\tau(0,s)=\tau(m+2,s)$, $\tau(t,0)=\rho(t)$ and $\tau(t,1)=x_0\in
X$ for any $s\in [0,1]$ and $t\in[0,m+2]$. Thus, $r\circ \tau$
provides a homotopy of the path $r\circ\rho:[0,m+2]\to\T$ and a
constant path. According to (w0), $w(r\circ\rho)=0$. Then by (w1),
$$
0=w(r\circ\rho)=\sum_{j=0}^{m+1}w\bigl(r\circ\rho\bigr|_{[j,j+1]}\bigr).
$$
Since $\gamma_0$ and $\gamma_1$ take values in $U\subseteq
r^{-1}(\T\setminus\{z_0\})$, $r\circ\rho\bigr|_{[0,1]}$ and
$r\circ\rho\bigr|_{[m+1,m+2]}$ take values in $\T\setminus\{z_0\}$.
According to (w4), $|w\bigl(r\circ\rho\bigr|_{[0,1]}\bigr)|<1$ and
$w\bigl(r\circ\rho\bigr|_{[m+1,m+2]}\bigr)<1$. Thus, by the last
display,
$$
\biggl|\sum_{j=1}^{m}w\bigl(r\circ\rho\bigr|_{[j,j+1]}\bigr)\biggr|<2.
$$
On the other hand, from (\ref{path}) and (\ref{gn1}) we see that for
each $j\in\{1,\dots,m\}$,
$$
r\circ\rho(t)=r(T^{n+j-1}\alpha(t-j))=z^{n+j-1}r(\alpha(t-j))=z^{n+j-1}\beta(t-j)
\ \ \text{for any $t\in[j,j+1]$}.
$$
Thus, according to (w2) and (w3),
$w\bigl(r\circ\rho\bigr|_{[j,j+1]}\bigr)=w(\beta)$ for $1\leq j\leq
m$. Hence
$$
2>\biggl|\sum_{j=1}^{m}w\bigl(r\circ\rho\bigr|_{[j,j+1]}\bigr)\biggr|=|mw(\beta)|=
m|w(\beta)|>2.
$$
This contradiction completes the proof.
\end{proof}

Now we are ready to prove Theorem~\ref{general}.

\begin{proof}[Proof of Theorem~\ref{general}]
We can disregard the case of one-element $X$ for its triviality. Let
$x\in Y$ and $N_x$ be the set defined in (\ref{nx}). By
Lemma~\ref{l2}, the set $H=F_{x,x}=\{h\in G:(x,h)\in N_x\}$ is a
closed subgroup of $G$. We shall show that $H=G$. By Lemma~\ref{l2},
there exists a separately continuous function $f:Y\times Y\to G/H$
such that $f(v,Ty)=gf(v,y)$ for any $v,y\in Y$. Consider the
function $\psi:Y\to G/H$, $\psi(y)=f(x,y)$. Then $\psi$ is
continuous and $\psi(Ty)=g\psi(y)=(gH)\psi(y)$ for any $y\in Y$.
Since $x\in Y\subseteq\uu(T)$, the orbit $O(T,x)$ is dense in $X$.
On the other hand, since $T(Y)\subseteq Y$, $O(T,x)\subseteq Y$ and
therefore $Y$ is dense in $X$. Since $X$ is not one-element, then so
is $Y$. Since $Y$ is connected, $Y$ has no isolated points. By
Lemma~\ref{ll1}, $x\in \overline{O(T,Tx)}\setminus O(T,Tx)$. Hence
$O(T,Tx)$ is not closed in $Y$. Lemma~\ref{map}, applied to the
restriction of $T$ to $Y$, implies that $gH=1_{G/H}$. On the other
hand, $g$ is a generator of $G$ and therefore $gH$ is a generator of
$G/H$. It follows that the group $G/H$ is trivial. That is, $H=G$.

By Lemma~\ref{l2}, for each $y\in Y$, the set $F_{x,y}=\{h\in
G:(y,h)\in N_x\}$ is a coset of $H$ in $G$. Since $H=G$, we have
$F_{x,y}=G$ for any $y\in Y$. Hence $Y\times G\subseteq N_x$. Since
$Y$ is dense in $X$, the last inclusion implies that $N_x=X\times
G$. Thus, $\{(T^nx,g^n):n\in\Z_+\}$ is dense in $X\times G$.
\end{proof}

\section{Connectedness \label{connect}}

Let $X$ be a topological vector space. We can consider the
corresponding projective space $\PP X$. As a set $\PP X$ consists of
one-dimensional linear subspaces of $X$. We define the natural
topology on $\PP X$ by declaring the map $\pi:X\setminus \{0\}\to\PP
X$, $\pi(x)=\langle x\rangle$ continuous and open, where $\langle
x\rangle$ is the linear span of the one-element set $\{x\}$. If
$X=\K^{n+1}$, then $\PP X$ is the usual $n$-dimensional (real or
complex) projective manifold. Note that if $T\in L(X)$ is injective,
it induces the continuous map $T_P:\PP X\to \PP X$, $T_P\langle
x\rangle=\langle Tx\rangle$. Finally, we can consider the space
$\PP_+ X$ of the rays in $X$. Namely, the elements of $\PP_+ X$ are
the rays $[x]=\{tx:t\geq 0\}$ for $x\in X\setminus \{0\}$ and the
topology on $\PP_+X$ is defined by declaring the map
$\pi_+:X\setminus \{0\}\to\PP_+ X$, $\pi(x)=[x]$ continuous and
open. Clearly, $\PP_+ X$ is homeomorphic to the unit sphere in $X$
if $X$ is a normed space.

In order to apply Theorem~\ref{general}, we need the following
lemmas.

\begin{lemma}\label{conn1}
Any topological vector space is path connected, locally path
connected and simply connected.
\end{lemma}

\begin{lemma}\label{conn2} If $X$ is a topological vector space
of real dimension $\geq3$, then $X\setminus\{0\}$ is path connected,
locally path connected and simply connected.
\end{lemma}

\begin{lemma}\label{conn3} If $X$ is a complex topological
vector space, then $\PP X$ is path connected, locally path connected
and simply connected.
\end{lemma}

\begin{lemma}\label{conn4} If $X$ is a topological
vector space of real dimension $\geq3$, then $\PP_+ X$ is path
connected, locally path connected and simply connected.
\end{lemma}

The proof of the above lemmas comes as a combination of well-known
facts and application of standard techniques of infinite dimensional
topology. We include the complete proofs for convenience of the
reader, but ban them to the last section.

\medskip

{\bf Remark.} \ It is worth noting that if $X$ is a real topological
vector space of dimension $\geq 2$, then $\PP X$ fails to be simply
connected.

\section{First applications of Theorem~\ref{general} \label{applic1}}

In this section, we derive a handful of known results and their
slight generalizations from Theorem~\ref{general}. We start with
proving of a few corollaries of Theorem~\ref{general}.

We need the following theorem, proved  in \cite{ww}, which is a
generalization to arbitrary topological vector spaces of a result
known previously for locally convex topological vector spaces.

\begin{thmw}Let $T$ be a continuous linear operator on an infinite
dimensional  topological vector space $X$. If $T$ is hypercyclic,
then $p(T)$ has dense range for any non-zero polynomial $p$. If $T$
is supercyclic and the space $X$ is complex, then the point spectrum
$\sigma_p(T^*)$ is either empty or a one-element set $\{z\}$ with
$z\in\C\setminus\{0\}$. If $\sigma_p(T^*)=\varnothing$, then
$p(T)(X)$ is dense in $X$ for each non-zero polynomial $p$. If
$\sigma_p(T^*)=\{z\}$, then $p(T)(X)$ is dense in $X$ for any
polynomial $p$ such that $p(z)\neq 0$. Finally, if $T$ is
supercyclic and the space $X$ is real, then anyway, there exists
$z_0\in\C\setminus\{0\}$ such that $p(T)(X)$ is dense in $X$ for any
real polynomial $p$ such that $p(z_0)\neq 0$.
\end{thmw}

\begin{corollary} \label{general1} Let $T$ be a hypercyclic
continuous linear operator on a topological vector space $X$ and $g$
be a generator of a compact topological group $G$. Then
$\{(T^nx,g^n):n\in\Z_+\}$ is dense in $X\times G$ for any $x\in
\uu(T)$. In other words, $\uu(T\oplus M_g)=\uu(T)\times G$.
\end{corollary}

\begin{proof} Let $\pp=\K[z]$ be the space of polynomials on
one variable. By Theorem~W, $p(T)$ has dense range for any $p\in
\pp\setminus\{0\}$. Thus $O(T,p(T)x)=p(T)(O(T,x))$ is dense for any
$x\in \uu(T)$. It follows that the set $\uu(T)$ is invariant under
$p(T)$ for any $p\in\pp\setminus\{0\}$. Let $x\in \uu(T)$ and
$X_0=\{p(T)x:p\in\pp\}$. Since each $p(T)x$ with non-zero $p$
belongs to $\uu(T)$, we see that the linear map $p\mapsto p(T)x$
from $\pp$ to $X_0$ is one-to-one and onto and the restriction of
$T$ to the invariant subspace $X_0$ is injective. According to
Lemma~\ref{conn2}, $Y=X_0\setminus \{0\}$ is path connected, locally
path connected and simply connected. Since $Y\subseteq \uu(T)$ and
$T(Y)\subseteq Y$, from Theorem~\ref{general} it follows that
$\{(T^nx,g^n):n\in\Z_+\}$ is dense in $X\times G$.
\end{proof}

Similar property can be proven for supercyclic operators. First, we
observe that supercyclicity of an injective operator $T\in L(X)$ can
be interpreted as universality of the induced map $T_P$,
$T_P(\langle x\rangle)=\langle Tx\rangle$ on the projective space
$\PP X$. Namely, $x\in X$ is supercyclic for $T$ if and only if
$\langle x\rangle$ is universal for $T_P$.

\begin{corollary} \label{gege1} Let $X$ be an infinite dimensional
complex topological vector space and $T\in L(X)$ be supercyclic.
Then for any supercyclic vector $x$ for $T$ and any generator $g$ of
a compact topological group $G$, $\{(zT^nx,g^n):n\in\Z_+,\ z\in\C\}$
is dense in $X\times G$.
\end{corollary}

\begin{proof} Let $x$ be a supercyclic vector for $T$.

{\bf Case} $\sigma_p(T^*)=\varnothing$. Consider the linear space
$X_0=\{p(T)x:p\in\pp\}$. Exactly as in the proof of
Corollary~\ref{general1}, we have that $X_0$ is infinite dimensional
and dense in $X$, the restriction of $T$ to the invariant subspace
$X_0$ is injective and each non-zero element of $X_0$ is a
supercyclic vector for $T$. It follows that each element of the
projective space $\PP X_0$ is universal for the induced map $T_P$.
By Lemma~\ref{conn3}, $\PP X_0$ is path connected, locally path
connected and simply connected. Theorem~\ref{general} implies that
$\{(T_P^n\langle x\rangle,g^n):n\in\Z_+\}$ is dense in $\PP
X_0\times G$ and therefore in $\PP X\times G$. The latter density
means that $\{(zT^nx,g^n):n\in\Z_+,\ z\in\C\}$ is dense in $X\times
G$.

{\bf Case} $\sigma_p(T^*)\neq\varnothing$. By Theorem~W, there
exists $w\in\C\setminus\{0\}$ such that $\sigma_p(T^*)=\{w\}$.
Multiplying $T$ by $w^{-1}$, we can, without loss of generality,
assume that $w=1$. Pick non-zero $f\in X^*$ such that $T^*f=f$.
Clearly $f(x)\neq 0$. Indeed, otherwise the orbit of $x$ lies in
$\ker f$ and $x$ is not even a cyclic vector for $T$. Multiplying
$f$ by a non-zero constant, if necessary, we may assume that
$f(x)=1$. Let $Z=\{x\in X:f(x)=1\}$. It is straightforward to verify
that $T(Z)\subseteq Z$ and $u\in Z$ is a supercyclic vector for $T$
if and only if $u$ is universal for the restriction $T\bigr|_Z$ of
$T$ to $Z$. Now, the set $Y=\{p(T)x:p\in\pp,\ p(1)=1\}$ is a subset
of $Z$, $x\in Y$ and $T(Y)\subseteq Y$. From Theorem~W it follows
that each $p(T)x$ with $p(1)\neq 0$ is a supercyclic vector for $T$.
Hence each element of $Y$ is universal for $T\bigr|_Z$. Moreover,
according to Lemma~\ref{conn1}, $Y$ being an affine subspace of $X$,
is path connected, locally path connected and simply connected. By
Theorem~\ref{general}, $\{(T^nx,g^n):n\in\Z_+\}$ is dense in
$Z\times G$. Hence $\{(zT^nx,g^n):n\in\Z_+,\ z\in\C\}$ is dense in
$X\times G$.
\end{proof}

It is worth noting that an analog of Corollary~\ref{gege1} fails for
supercyclic operators on real topological vector spaces. Indeed, let
$t\in\R\setminus\Q$, $A$ be a linear operator on $\R^2$ with the
matrix
$$
A=\begin{pmatrix}\cos 2\pi t&\sin 2\pi t\\
-\sin 2\pi t&\cos 2\pi t\end{pmatrix}
$$
and $B$ be any hypercyclic operator on a real topological vector
space $X$. It is easy to verify that $T=A\oplus B$ is a supercyclic
operator on $\R^2\times X$, $g=e^{2\pi it}$ is a generator of $\T$,
while $\{(zT^nx,g^n):n\in\Z_+,\ z\in\R\}$ is not dense in
$(\R^2\times X)\times \T$ for any $x\in\R^2\times X$. On the other
hand if we impose the additional condition that $p(T)$ has dense
range for any non-zero polynomial $p$, the result extends to the
real case. One has to use the fact that a vector is $\R$-supercyclic
if and only if it is $\R_+$-supercyclic \cite{bermud} and apply
Theorem~\ref{general} exactly as in the first case of the above
proof with $\PP X_0$ replaced by $\PP_+ X_0$ and $T_P$ by
$T_{P_+}[x]=[Tx]$. Of course, one has to use Lemma~\ref{conn4}
instead of Lemma~\ref{conn3} to ensure the required connectedness.
This leads to the following corollary.

\begin{corollary} \label{gege2} Let $X$ be a real topological vector
space and $T\in L(X)$ be a supercyclic operator such that $p(T)$ has
dense range for any non-zero $p\in\pp$. Then for any supercyclic
vector $x$ for $T$ and any generator $g$ of a compact topological
group $G$, $\{(sT^nx,g^n):n\in\Z_+,\ s>0\}$ is dense in $X\times G$.
\end{corollary}

\subsection{The Ansari theorem}

Applying Corollary~\ref{general1} in the case when $g$ is a
generator of an $n$-element cyclic group $G$, carrying the discrete
topology, we immediately obtain the following statement, which is
the Ansari theorem \cite{ansa} on hypercyclicity of powers of
hypercyclic operators.

\begin{corollary} \label{general2} Let $X$ be a topological vector
space, $T\in L(X)$ and $n\in\N$. Then $\uu(T^n)=\uu(T)$. In
particular $T^n$ is hypercyclic if and only if $T$ is hypercyclic.
\end{corollary}

The supercyclicity version of the Ansari theorem for operators on a
complex topological vector space follows similarly from
Corollary~\ref{gege1}. In order to incorporate the real case, we
need the following non-linear analog of the Ansari theorem. The
advantage compared to the direct application of
Theorem~\ref{general} is that in the case of finite group $G$, one
can significantly relax the topological assumptions on the set of
universal elements.

\begin{proposition}\label{l1} Let $X$ be a topological space with no
isolated points, $n\in\N$ and $T:X\to X$ be a continuous map such
that $\uu(T)$ is connected. Then $\uu(T^n)=\uu(T)$.
\end{proposition}

\begin{proof} Let $g$ be a generator of the cyclic group $G$ of
order $n$ (carrying the discrete topology) and $x\in\uu(T)$. By
Lemma~\ref{l2} the set $H=F_{x,x}$ defined in \ref{nx} is a subgroup
of $G$, each $F_{x,y}$ for $y\in\uu(T)$ is a coset of $H$, the map
$\psi:\uu(T)\to G/H$, $\psi(y)=F_{x,y}$ is continuous and
$\psi(T^kx)=g^kH$. Since $g$ is a generator of $G$, it follows that
$\psi$ is onto. Since $G/H$ is finite, $\psi$  is continuous and
$\uu(T)$ is connected, it follows that $G/H$ is trivial. That is,
$H=G$. Hence $F_{x,y}=H$ for any $y\in\uu(T)$ and therefore
$\uu(T)\times G$ is contained in the closure of the orbit $O(T\oplus
M_g,(x,1))$. From Lemma~\ref{ll1} we see that $\uu(T)$ contains
$O(T,x)$ and therefore is dense in $X$. Hence $(x,1)$ is universal
for $T\oplus M_g$. Since $G$ carries the discrete topology,
$\{T^kx:g^k=1\}$ is dense in $X$. The latter set is exactly
$O(T^n,x)$. Thus $x\in \uu(T^n)$. That is, $\uu(T)\subseteq
\uu(T^n)$. The opposite inclusion is obvious.
\end{proof}

It is worth noting that the original proof of Ansari \cite{ansa} can
also be adapted to prove the last proposition. Our point was to show
that the result follows also from Lemma~\ref{l2}.

\begin{corollary} \label{general22} Let $X$ be a topological vector
space, $T\in L(X)$ and $n\in\N$. Then the sets of supercyclic
vectors for $T^n$ and $T$ coincide. In particular $T^n$ is
supercyclic if and only if $T$ is supercyclic.
\end{corollary}

\begin{proof} It is well-known that if $X$ is finite dimensional,
then supercyclic operators do exist on $X$ if and only if the real
dimension of $X$ does not exceed $2$. In this case we can easily see
that the required  property is satisfied. Thus we can assume that
$X$ is infinite dimensional. Consider the space
$X_0=\{p(T)x:p\in\pp\}$. Since $X_0$ is dense in $X$, it is infinite
dimensional. Moreover, the restriction of $T$ to $X_0$ is injective.
By Theorem~W, there exists $z_0\in\C$ such that the set
$\{p(T)x:p(z_0)\neq 0\}$ consists of supercyclic vectors for $T$.
Thus the set $Y=\{\langle p(T)x\rangle: p(z_0)\neq 0\}$ consists of
universal elements of the induced map $T_P:\PP X_0\to \PP X_0$. One
can easily see that $Y$ is connected and dense in $\PP X_0$. Hence
$\uu(T_P)$ is connected. By Proposition~\ref{l1}, $\langle
x\rangle\in \uu(T_P^n)$. It follows that $x$ is a supercyclic vector
for $T^n$.
\end{proof}

\subsection{The Le\'on--M\"uller theorem}

Let $X$ be a topological vector space. Recall that the pointwise
convergence topology on $L(X)$ is called the {\it strong operator
topology}. We say that an operator $S\in L(X)$ is {\it a generalized
rotation} if the strong operator topology closure $G$ of
$\{S^n:n\in\Z_+\}$ in $L(X)$ is a compact {\bf subgroup} of the
semigroup $L(X)$ and the map $(x,A)\mapsto Ax$ from $X\times G$ to
$X$ is continuous.

\begin{proposition}\label{lm00} Let $X$ be a topological vector
space, $T\in L(X)$ and $S\in L(X)$ be a generalized rotation. Then
$T$ is hypercyclic if and only if $\{S^nT^n:n\in\Z_+\}$ is universal
and $\uu(T)\subseteq\uu(\{S^nT^n:n\in\Z_+\})$. If additionally
$TS=ST$, then $T$ is hypercyclic if and only if $ST$ is hypercyclic
and $\uu(T)=\uu(ST)$.
\end{proposition}

\begin{proof}Let $G$ be the strong operator topology closure of
$\{S^n:n\in\Z_+\}$ in $L(X)$. Since $S$ is a generalized rotation,
$G$ is compact and the map $\Phi:X\times G\to X$, $\Phi(x,A)=Ax$
from $X\times G$ to $X$ is continuous. Let $x\in \uu(T)$. By
Corollary~\ref{general1}, the set $\Omega=\{(T^nx,S^n):n\in\Z_+\}$
is dense in $X\times G$. Since $\Phi$ is onto and continuous
$\Phi(\Omega)=\{S^nT^nx:n\in\Z_+\}$ is dense in $X$. Hence $x\in
\uu(\{S^nT^n:n\in\Z_+\})$. Thus $\uu(T)\subseteq
\uu(\{S^nT^n:n\in\Z_+\})$.

Assume now that $TS=ST$. Then $S^nT^n=(ST)^n$ for each $n\in\Z_+$.
Hence $\uu(T)\subseteq \uu(ST)$. Clearly $S$ is invertible and
$S^{-1}$ is also a generalized rotation. Hence $\uu(ST)\subseteq
\uu(S^{-1}(ST))=\uu(T)$. Thus $\uu(T)=\uu(ST)$.
\end{proof}

If $X$ is a complex topological vector space and $z\in\T$, then
obviously, the rotation operator $zI$ is a generalized rotation.
Applying Proposition~\ref{lm00} with $S=zI$, we immediately obtain
the Le\'on-Saavedra and M\"uller theorem \cite{muller} on
hypercyclicity of rotations of hypercyclic operators.

\begin{corollary} \label{general3} If $T$ is a hypercyclic continuous linear
operator on a complex topological vector space $X$ and $z\in\T$,
then $zT$ is hypercyclic and $\uu(zT)=\uu(T)$. \end{corollary}

If a $X_1,\dots,X_n$ are complex topological vector spaces,
$X=X_1\times{\dots}\times X_n$ and $z_1,\dots,z_n\in\T$, then the
operator $S\in L(X)$, $S=z_1I\oplus{\dots}\oplus z_nI$ is a
generalized rotation. Applying Proposition~\ref{lm00} with this $S$,
we get the following slight generalization of the Le\'on--M\"uller
theorem.

\begin{proposition}\label{lm0} Let $T_j\in L(X_j)$ for $1\leq j\leq n$ be such
that $T=T_1\oplus{\dots}\oplus T_n$ is hypercyclic. Then for any
$z=(z_1,\dots,z_n)\in \T^n$, the operator
$T_z=z_1T_1\oplus{\dots}\oplus z_nT_n$ is also hypercyclic and
$\uu(T)=\uu(T_z)$.
\end{proposition}

Similar proof with application of Corollary~\ref{gege1} instead of
Corollary~\ref{general1} gives the following supercyclic analog of
Proposition~\ref{lm00}.

\begin{proposition}\label{lm000} Let $X$ be a complex topological vector
space, $T\in L(X)$ and $S\in L(X)$ be a generalized rotation. Then
$T$ is supercyclic if and only if ${\cal F}=\{sS^nT^n:n\in\Z_+,\
s\in\C\}$ is universal and any supercyclic vector for $T$ is
universal for $\cal F$. If additionally $TS=ST$, then $T$ is
supercyclic if and only if $ST$ is supercyclic and the operators $T$
and $ST$ have the same sets of supercyclic vectors.
\end{proposition}

Applying Proposition~\ref{lm000} with $S=z_1I\oplus{\dots}\oplus
z_nI$, we also get the following corollary.

\begin{proposition}\label{lm1} Let $X_j$ be complex topological vector
spaces and $T_j\in L(X_j)$ for $1\leq j\leq n$ be such that
$T=T_1\oplus{\dots}\oplus T_n$ is supercyclic. Then for any
$z=(z_1,\dots,z_n)\in \T^n$, the operator
$T_z=z_1T_1\oplus{\dots}\oplus z_nT_n$ is also supercyclic and the
sets of supercyclic vectors for $T$ and $T_z$ coincide.
\end{proposition}

{\bf Remark.} \ The same argument allows to replace finite direct
sums of operators in Propositions~\ref{lm0} and~\ref{lm1} by, for
instance, countable $c_0$ or $\ell_p$ ($1\leq p<\infty$) sums of
Banach spaces.

\subsection{Hypercyclicity of members of a continuous universal semigroup}

Recently, Conejero, M\"uller and Peris \cite{semi} have proven that
if $\{T_t\}_{t\geq 0}$ is a strongly continuous semigroup of
continuous linear operators acting on a complete metrizable
topological vector space $X$ and the family $\{T_t:t\in\R_+\}$ is
universal, then each $T_t$ with $t>0$ is hypercyclic. We show now
that this theorem is also a corollary of Theorem~\ref{general}.

\begin{proposition} \label{sem} Let $\{T_t\}_{t\geq 0}$ be semigroup of
continuous linear operators acting on a topological vector space
$X$. Assume also that the map $(t,x)\mapsto T_tx$ from $\R_+\times
X$ to $X$ is continuous. Then $\uu(T_t)=\uu(T_s)$ for any $t,s>0$.
That is, either $T_t$ for $t>0$ are all non-hypercyclic or $T_t$ for
$t>0$ are all hypercyclic and have the same set of hypercyclic
vectors.
\end{proposition}

\begin{proof} Let $t,s>0$. We have to show that
$\uu(T_t)=\uu(T_s)$. Clearly, it is enough to demonstrate that
$\uu(T_s)\subseteq \uu(T_t)$ for any $s,t>0$. If $t/s\in \Q$, then
$t/s=m/n$ for some $m,n\in\N$. Then $T_t^n=T_s^m$ and according to
the Ansari theorem (Corollary~\ref{general2}), we have
$\uu(T_t)=\uu(T_t^n)=\uu(T_s^m)=\uu(T_s)$. It remains to consider
the case $t/s\notin\Q$. In this case $z=e^{2\pi is/t}$ is a
generator of $\T$. Let $x\in\uu(T_s)$ and $y\in X$. By
Corollary~\ref{general1}, the set $\{(T_{ns}x,z^n):n\in\Z_+\}$ is
dense in $X\times \T$. Hence, we can pick a net
$\{n_\alpha\}_{\alpha\in D}$ of non-negative integers such that
$T_{n_\alpha s}x\to y$ and $z^{n_\alpha}\to e^{-2\pi is}$. Since
$z^{n_\alpha}=e^{2\pi in_\alpha s/t}$ and $z^{n_\alpha}\to e^{-2\pi
is}$, we can pick a net $\{m_\alpha\}_{\alpha\in D}$ of non-negative
integers such that $s(n_\alpha+1)-tm_\alpha\to 0$. That is,
$tm_\alpha=sn_\alpha+b_\alpha$, where $\{b_\alpha\}$ is a net in
$\R_+$, converging to $s$. Using continuity of the map $(a,x)\mapsto
T_ax$, we see that $T_{tm_\alpha}x=T_{b_\alpha}T_{sn_\alpha}x\to
T_sy$ since $b_\alpha\to s$ and $T_{sn_\alpha}x\to y$. Hence
$T_sy\in \overline{O(T_t,x)}$ for any $y\in X$. Since $T_s$ is
hypercyclic, it has dense range and therefore $O(T_t,x)$ is dense in
$X$ and $x\in\uu(T_t)$. Thus, $\uu(T_s)\subseteq\uu(T_t)$.
\end{proof}

The following observation belongs to Oxtoby and Ulam \cite{ou}. We
present it in a slightly more general form, although the proof does
not need any changes.

\begin{proposition} \label{pou} Let $M$ be a Baire separable
metrizable space and $\{T_t\}_{t\geq 0}$ be a semigroup of
continuous maps from $M$ to itself such that the map $(t,x)\mapsto
T_tx$ from $\R_+\times M$ to $M$ is continuous and the family
$\{T_t:t\in\R_+\}$ is universal. Then for almost all $t\in\R_+$ in
the Baire category sense, $T_t$ is universal. In particular, there
exists $t>0$ such that $T_t$ is universal.
\end{proposition}

Taking into account that the map $(t,x)\mapsto T_tx$ is continuous
for any strongly continuous semigroup $\{T_t\}_{t\geq 0}$ of
continuous linear operators on a Baire topological vector space and
combining Proposition~\ref{pou} with Proposition~\ref{sem}, we
immediately obtain the following corollary, which is exactly the
Conejero--M\"uller--Peris theorem.

\begin{corollary}\label{0000} Let $\{T_t\}_{t\geq 0}$ be a strongly
continuous semigroup of continuous linear operators on a Baire
separable metrizable topological vector space such that the family
$\{T_t:t\in\R_+\}$ is universal. Then each $T_t$ for $t>0$ is
hypercyclic and $T_t$ for $t>0$ share the set of hypercyclic
vectors.
\end{corollary}

\section{A characterization of $\R_+$-supercyclic operators \label{rsup}}

The proof of Theorem~\ref{rp} naturally splits in two cases: when
$\sigma_p(T^*)$ is empty and when $\sigma_p(T^*)$ is a singleton.
The case $\sigma_p(T)=\varnothing$ is covered by the following
proposition by Le\'on-Saavedra and M\"uller \cite{muller}. They
prove it for operators on Banach spaces, however, virtually the same
proof works for operators acting on arbitrary topological vector
spaces (just replace sequence convergence by net convergence).

\begin{thmlm} Let $X$ be a complex topological vector
space and $T\in L(X)$ be a supercyclic operator such that
$\sigma_p(T^*)=\varnothing$. Then $T$ is $\R_+$ supercyclic.
Moreover the sets of supercyclic and $\R_+$-supercyclic vectors for
$T$ coincide.
\end{thmlm}

It remains to consider the case when $\sigma_p(T^*)$ is a singleton.

\begin{proposition}\label{rp2} Let $X$ be a complex topological vector
space and $T\in L(X)$ be a supercyclic operator such that
$\sigma_p(T^*)=\{z\}$ for some $z\in\C\setminus\{0\}$ such that
$z/|z|$ has infinite order in $\T$. Then $T$ is $\R_+$ supercyclic.
Moreover the sets of supercyclic and $\R_+$-supercyclic vectors for
$T$ coincide.
\end{proposition}

\begin{proof} Since any $\R_+$-supercyclic vector for $T$ is
supercyclic, it suffices to verify that any supercyclic vector for
$T$ is $\R_+$-supercyclic. Replacing $T$ be $|z|^{-1}T$, if
necessary, we can, without loss of generality, assume that $|z|=1$.
Let $x\in X$ be a supercyclic vector for $T$. It suffices to show
that $x$ is $\R_+$-supercyclic for $T$. Let $f\in X^*$ be a non-zero
functional such that $T^*f=zf$ and $f(x)=1$ (the last condition is
just a normalization). Consider the affine hyperplane $Z=\{u\in
X:f(u)=1\}$ and a map $S:Z\to Z$, $S(u)=z^{-1}Tu$. The map $S$ is
indeed taking values in $Z$ since
$f(Su)=z^{-1}f(Tu)=z^{-1}(T^*f)(u)=z^{-1}zf(u)=f(u)=1$ for any $u\in
Z$. It is also clear that $S$ is continuous. Let
$\pp_1=\{p\in\pp:p(z)=1\}$ and $Y=\{p(T)x:p\in\pp_1\}$. First, we
shall demonstrate that $Y$ is a dense subset of $Z$, $S(Y)\subseteq
Y$ and $Y\subseteq\uu(S)$. Indeed, let $y\in Y$. Then $y=p(T)x$ for
some $p\in \pp_1$. Since $f(y)=f(p(T)x)=(p(T^*)f)(x)=p(z)f(x)=1$, we
have $y\in Y$ and therefore $Y\subseteq Z$. Next,
$Sy=z^{-1}Ty=q(T)x$, where $q(t)=z^{-1}tp(t)$. Since
$q(z)=z^{-1}zp(z)=p(z)=1$, we have $Sy\in Y$ and therefore
$S(Y)\subseteq Y$. Next, since $x$ is supercyclic for $T$, the set
$A=\{sT^nx:n\in\Z_+,\ s\in\C\setminus\{0\}\}$ is dense in $X$. By
Theorem~W, $p(T)(A)=\{sT^ny:n\in\Z_+,\ s\in\C\setminus\{0\}\}$ is
also dense in $X$. Since $A\subset X\setminus\ker f$ and the map
$\Phi:X\setminus \ker f\to Z$, $\Phi(y)=y/f(y)$ is continuous and
onto, we have that $\Phi(p(T)(A))$ is dense in $Z$. On the other
hand, it is easy to see that $\Phi(p(T)(A))=O(S,y)$. Thus $y\in
\uu(S)$ and therefore $Y\subseteq \uu(S)$. Now, $Y$ is an affine
subspace of $X$ and therefore, Lemma~\ref{conn1} implies that $Y$ is
path connected, locally path connected and simply connected. Since
$z$ has infinite order in $\T$, $z$ is a generator of $\T$. Applying
Theorem~\ref{general}, we see that $\{(S^ny,z^n):n\in\Z_+\}$ is
dense in $Z\times \T$ for any $y\in Y$. In particular,
$\{(S^nx,z^n):n\in\Z_+\}$ is dense in $Z\times \T$. By definition of
$S$, $\{(z^{-n}T^nx,z^n)\}$ is dense in $Z\times \T$. Consider the
map $\Psi:Z\times \T\to X$, $\Psi(u,s)=su$. Clearly $\Psi$ is
continuous and $\Psi(Z\times\T)=Z_0=\{u\in X:|f(u)|=1\}$. Therefore
the set $\Psi(\{(z^{-n}T^nx,z^n):n\in\Z_+\})=\{T^nx:n\in\Z_+\}$ is
dense in $Z_0$. Since $\{su:s\in\R_+,\ u\in Z_0\}$ is dense in $X$,
we see that $\{sT^nx:n\in\Z_+,\ s\in\R_+\}$ is also dense in $X$.
That is, $x$ is an $\R_+$-supercyclic vector for $T$.
\end{proof}

\begin{proof}[Proof of Theorem~\ref{rp}] If
$\sigma_p(T^*)=\varnothing$ or $\sigma_p(T^*)=\{z\}$ with $z/|z|$
being an infinite order element of $\T$, then according to
Propositions~LM and~\ref{rp2}, $T$ is $\R_+$-supercyclic. It remains
to show that if $\sigma_p(T^*)=\{z\}$ with $z\in\C\setminus \{0\}$
and $z/|z|$ having finite order in $\T$, then $T$ is not
$\R_+$-supercyclic. Replacing $T$ be $|z|^{-1}T$, if necessary, we
can, without loss of generality, assume that $|z|=1$. Then $z\in\T$
has finite order. Let $G$ be the finite subgroup of $\T$ generated
by $z$. Pick a non-zero $f\in X^*$ such that $T^*f=zf$. Let $x\in
X$. Then for any $s\in\R_+$ and $n\in\Z_+$, $f(sT^nx)=sz^nf(x)\in
A=f(x)G\R_+$. The set $A$, being a finite union of rays, is nowhere
dense in $\C$. Since $f:X\to\C$ is open, the set $f^{-1}(A)$ is
nowhere dense in $X$. Hence $\{sT^nx:n\in\Z_+,\ s\in\R_+\}$, being a
subset of $f^{-1}(A)$, is nowhere dense in $X$. Thus $x$ is not an
$\R_+$-supercyclic vector for $T$. Since $x\in X$ is arbitrary, $T$
is not $\R_+$-supercyclic.
\end{proof}

\section{Cyclic direct sums of scalar multiples of a
supercyclic operator \label{dirsum}}

We shall repeatedly use the following elementary observation.

\begin{lemma}\label{elem} Let $T$ be a continuous linear operator on a
topological vector space $X$, $x\in X$ and $L$ be a closed linear
subspace of $X$ such that $T(L)\subseteq L$, $L$ is contained in the
cyclic subspace $C(T,x)=\overline{\spann}\{T^kx:k\in\Z_+\}$ and
$x+L$ is a cyclic vector for the quotient operator $\widetilde T\in
L(X/L)$, $\widetilde T(u+L)=Tu+L$. Then $x$ is a cyclic vector for
$T$.
\end{lemma}

\begin{proof} Let $U$ be a non-empty open subset of $X$ and
$\pi:X\to X/L$ be the canonical map $\pi(x)=x+L$. Since $x+L$ is a
cyclic vector for $\widetilde T$, there exists $p\in\pp$ such that
$p(\widetilde T)(x+L)\in \pi(U)$. Hence we can pick $w\in L$ such
that $w+p(T)x\in U$. That is, $w$ belongs to the open set
$-p(T)x+U$. Since $w\in L\subseteq C(T,x)$, we can find $q\in \pp$
such that $q(T)x\in -p(T)x+U$. That is, $(p+q)(T)x\in U$. Thus
$\{r(T)x:r\in\pp\}$ is dense in $X$ and therefore $x$ is a cyclic
vector for $T$.
\end{proof}

The proof of Theorem~\ref{zo} is different in the cases $\K=\R$ and
$\K=\C$. We start with the more difficult complex case.

\subsection{Proof of Theorem~\ref{zo}. Case $\K=\C$}

\begin{lemma}\label{l3} Let $n\in\N$, $z=(z_1,\dots,z_n)\in\T^n$, $G$ be the
closure in $\T^n$ of $\{z^m:m\in\Z_+\}$, $T$ be s supercyclic
operator on a complex topological vector space $X$, $u$ be a
supercyclic vector for $T$, $x=(u,\dots,u)\in X^n$ and
$S=z_1T\oplus{\dots}\oplus z_nT$. Then $\overline{A}=B$, where
$$
A=\{sS^k x:k\in\Z_+,\ s\in\C\}\ \ \text{and}\ \
B=\{(w_1y,\dots,w_ny):y\in X,\ w\in G\}.
$$
\end{lemma}

\begin{proof} By Lemma~\ref{sub}, $G$ is a closed subgroup of $\T^n$. By
Corollary~\ref{gege1}, the set $C=\{(sT^ku,z^k):k\in\Z_+,\ s\in\C\}$
is dense in $X\times G$. Consider the map
$$
\phi:X\times G\to B,\quad \phi(v,w)=(w_1v,\dots,w_nv).
$$
Clearly $\phi$ is continuous and onto. Hence $\phi(C)=A$ is dense in
$B$. Thus, $\overline{A}=B$.
\end{proof}

\begin{corollary}\label{l4} Let $n,k\in\N$, $k\leq n$ and
$z=(z_1,\dots,z_n)\in\C^n$ be such that $|z_j|=1$ for $1\leq j\leq
k$ and $|z_j|<1$ if $j>k$, $z'=(z_1,\dots,z_k)\in\T^k$ and $G$ be
the closure in $\T^k$ of $\{(z')^m:m\in\Z_+\}$, $T$ be s supercyclic
operator on a complex topological vector space $X$, $u$ be a
supercyclic vector for $T$, $x=(u,\dots,u)\in X^n$ and
$S=z_1T\oplus{\dots}\oplus z_nT$. Then $\overline{A}\supseteq B$,
where
$$
A=\{sS^k x:k\in\Z_+,\ s\in\C\}\ \ \text{and}\ \
B=\{(w_1y,\dots,w_ky,0,\dots,0):y\in X,\ w\in G\}.
$$
If additionally, the numbers $z_1,\dots,z_k$ are pairwise different,
then the cyclic subspace $C(S,x)$ contains the space
$$
L_k=\{v\in X^n:v_j=0\ \ \text{for}\ \ j>k\}.
$$
\end{corollary}

\begin{proof} The inclusion $\overline{A}\supseteq B$ follows
immediately from Lemma~\ref{l3} and the inequalities $|z_j|<1$ for
$j>k$. Assume now that $z_1,\dots,z_k$ are pairwise different.
Clearly $\overline{A}\subseteq C(S,x)$. Hence $B\subseteq C(S,x)$.
Since $C(S,x)$ is a linear subspace of $X^n$, $\spann(B)\subseteq
C(S,x)$. It remains to demonstrate that $\spann(B)\supseteq L_k$.
Let $j\in \{1,\dots,n\}$, $v\in X$ and $v^j\in X^n$ be defined as
$v^j_j=v$ and $v^j_l=0$ for $l\neq j$. Since vectors $\{v^j:v\in X,\
1\leq j\leq k\}$ span $L_k$, it suffices to show that $v^j\in
\spann(B)$. Let $e^j\in\C^k$ be the $j$th basic vector: $e^j_j=1$
and $e^j_l=0$ for $l\neq j$. Since $z_1,\dots,z_k$ are pairwise
different, the matrix $\{z_j^l\}_{j,l=1}^k$ is invertible: its
determinant is the Van der Monde one. Hence, there exist
$\alpha_1,\dots,\alpha_k\in\C$ such that
$$
e^j=\sum\limits_{j=1}^k\alpha_j(z')^j\ \ \text{and therefore}\ \
v^j=\sum\limits_{j=1}^k\alpha_j(z_1^jv,\dots z_k^jv).
$$
Since each $(z_1^jv,\dots z_k^jv)$ belongs to $B$, we see that
$v^j\in\spann(B)$. Thus, $L_k\subseteq \spann(B)$.
\end{proof}

\begin{proof}[Proof of Theorem~\ref{zo}. Case $\K=\C$] Let $n\in\N$, $T$ be a
supercyclic continuous linear operator on a complex topological
vector space $X$, $u$ be a supercyclic vector for $T$ and
$z_1,\dots,z_n$ be pairwise different non-zero complex numbers. Let
also $S=z_1T\oplus{\dots}\oplus z_nT$. It suffices to show that
$x=(u,\dots,u)$ is a cyclic vector for $S$. We shall use induction
with respect to $n$. The case $n=1$ is trivial. Without loss of
generality, we may assume that $1=|z_1|\geq |z_2|\geq {\dots}\geq
|z_n|$. Let $m\in\{1,n\}$ be the maximal number for which
$|z_1|={\dots}=|z_m|$. Then $z_1,{\dots},z_m$ are pairwise different
elements of $\T$. By Corollary~\ref{l4}, the space $L_m=\{v\in
X^n:v_j=0\ \ \text{for}\ \ j>m\}$ is contained in the cyclic
subspace $C(S,x)$. It is also clear that $S(L_m)\subseteq L_m$ and
$X^n/L_m$ is naturally isomorphic to $X^{n-m}$ and the quotient
operator $\widetilde S(v+L_m)=Sv+L_m$, acting on $X^n/L_m$ is
naturally similar to $z_{m+1}T_{m+1}\oplus{\dots}\oplus z_nT_n$.
From the induction hypothesis it follows that $x+L_m$ is a cyclic
vector for $\widetilde S$. According to Lemma~\ref{elem}, $x$ is a
cyclic vector for $S$. \end{proof}

\subsection{Proof of Theorem~\ref{zo}. Case $\K=\R$}

\begin{lemma}\label{re1} Let $0<t_1<{\dots}<t_n$ and $T$ be a
supercyclic operator on a real topological vector space $X$. Then
$S=t_1T\oplus{\dots}\oplus t_nT$ is cyclic.
\end{lemma}

\begin{proof} Without loss of generality, we can assume that
$t_n=1$. Let $x$ be a supercyclic vector for $T$. We shall
demonstrate that $u=(x,\dots,x)$ is a cyclic vector for $S$. We use
the induction with respect to $n$. The case $n=1$ is trivial. Assume
that $n\geq 2$ and the statement is true for a sum $n-1$ positive
scalar multiples of a supercyclic operator. Since $T$ is supercyclic
and $t_j<1$ for $j<n$, we see that
$$
L=\{0\}\times{\dots}\times\{0\}\times X\subset
\overline{\{sS^nu:s\in\R,\ n\in\Z_+\}}.
$$
In particular, $L$ is contained in the cyclic subspace $C(S,u)$.
Obviously $S(L)\subseteq L$ and $X^n/L$ is naturally isomorphic to
$X^{n-1}$ and the quotient operator $\widetilde S(v+L)=Sv+L$, acting
on $X^n/L$ is naturally similar to $t_1T\oplus{\dots}\oplus
t_{n-1}T$. From the induction hypothesis it follows that $u+L_m$ is
a cyclic vector for $\widetilde S$. According to Lemma~\ref{elem},
$u$ is a cyclic vector for $S$.
\end{proof}

In order to incorporate negative multiples, we need the following
elementary observation.

\begin{lemma} \label{re2} Let $X$ be a topological vector space and
$T\in L(X)$ be such that $T(X)$ is dense in $X$ and $T^2$ is cyclic.
Then $S=T\oplus (-T)$ is cyclic.
\end{lemma}

\begin{proof} Let $x$ be a cyclic vector for $T^2$. We
shall demonstrate that $(x,x)$ is a cyclic vector for $S$. Indeed,
for any $p\in\pp$, $p(S^2)(x,x)=(p(T^2)x,p(T^2)x)$ and
$Sp(S^2)(x,x)=(Tp(T^2)x,-Tp(T^2)x)$. Since $T^2$ is cyclic and $T$
has dense range, we see that the cyclic space $C(S,(x,x))$ contains
the spaces $L_0=\{(u,u):u\in X\}$ and $L_1=\{(u,-u):u\in X\}$. Since
$X\times X=L_0\oplus L_1$, we see that $C(S,(x,x))=X\times X$ and
therefore $(x,x)$ is a cyclic vector for $S$.
\end{proof}

\begin{proof}[Proof of Theorem~\ref{zo}. Case $\K=\R$] Let
$A=\{|z_j|:1\leq j\leq n\}$ and $t_1,\dots,t_k$ be such that
$t_1<{\dots}<t_k$ and $A=\{t_1,\dots,t_k\}$. Since $z_j$ are real,
we see that $\{z_1,\dots,z_n\}\subseteq
\{t_1,\dots,t_k,-t_1,\dots,-t_k\}$. Thus, it is enough to show that
$R_0=R\oplus (-R)$ is cyclic, where $R=t_1T\oplus {\dots}\oplus
t_kT$. The Ansari theorem (Corollary~\ref{general22}) implies that
$T^2$ is supercyclic. By Lemma~\ref{re1}, the operator
$R^2=t_1^2T^2\oplus{\dots}\oplus t_n^2T^2$ is cyclic. Since $T$ is
supercyclic, the range of $T$ is dense and therefore the range of
$R$ is dense. By Lemma~\ref{re2}, $R_0=R\oplus(-R)$ is cyclic.
\end{proof}

\section{Concluding remarks \label{remarks}}

It would be interesting to investigate possibilities of extension of
Lemma~\ref{L2} to the case of non-commutative group $G$. The latter
could be useful in studying of universality of non-commuting
families of operators.

\subsection{Universal semigroups}

The Conejero--M\"uller--Peris theorem on hypercyclicity of members
of a strongly continuous universal semigroup of linear operators on
a complete metrizable topological vector space fails for semigroups
labeled by $\R_+^n$ for $n\geq 2$.

\begin{example}\label{ex1}
Take a compact weighted backward shift $S$ on $\ell_2$. The operator
$e^{tS}$ is hypercyclic for each $t\in\R$, $t\neq 0$ (see, for
instance, \cite{grish}). Take $c>\|e^S\|$ and consider the operators
$T_{t,s}=c^{s-t}e^{tS}\oplus c^{s-t}I$ acting on $\ell_2\oplus \K$.
From hypercyclicity of $S$ it follows that the family
$\{T_{t,s}:t\in\Z_+,\ s\in\R_+\}$ is universal and therefore the
strongly continuous semigroup $\{T_{t,s}\}_{t,s\in\R_+^2}$ is
universal. On the other hand, $T_{t,s}-c^{s-t}I$ has non-dense range
for any $(t,s)\in\R_+$. By Theorem~W each $T_{t,s}$ is
non-hypercyclic. It is also easy to see that
$\{T_{t,s}:(t,s)\in\Z_+^2\}$ is non-universal.
\end{example}

Thus, members of a universal strongly continuous semigroup
$\{T_t\}_{t\in\R_+^2}$ may be all non-hypercyclic and
$\{T_t\}_{t\in\Z_+^2}$ may be non-universal. This means that in
order to extend the theorem on hypercyclicity of members of a
strongly continuous universal semigroup, we need some extra
assumptions about the semigroup. The next question seems to be
natural and interesting.

\begin{question} \label{q1} Let $X$ be a separable infinite dimensional
complex Banach space and $z\mapsto T_z$ be a holomorphic operator
valued function from $\C$ to $L(X)$ such that $T_0=I$ and
$T_{z+w}=T_zT_w$ for any $z,w\in\C$. Assume also that the family
$\{T_z:z\in\C\}$ is universal. Is it true that each $T_z$ with
$z\neq 0$ is hypercyclic? What about holomorphic  semigroups,
labeled by a sector in $\C$?
\end{question}

Holomorphic universal semigroups of bounded linear operators on
Banach spaces have been treated in \cite{holo}, where it is shown
that every complex separable infinite-dimensional Banach space $X$
supports a holomorphic uniformly continuous mixing semigroup
$\{T_z\}_{z\in\Delta(\frac{\pi}{2})}$, where $\Delta(\frac{\pi}{2})$
is the sector $\{re^{i\phi}:r\geq 0,\ |\phi|<\frac{\pi}{2}\}$, and
$\{T_z\}_{z\in\Delta(\frac{\pi}{2})}$ is said to be (topologically)
mixing if, for any pair $(U,V)$ of nonempty open subsets of $X$,
there is $r_0>0$ such that $T(z)(U)\cap V\neq\varnothing$ as soon as
$|z|>r_{0}$. It is worth noting that $\{T_z\}$ is mixing if and only
if $\{T_{z_n}:n\in\N\}$ is universal for any sequence $\{z_n\}$ such
that $|z_n|\to\infty$. It is possible to show that the above result
admits a stronger form. Namely, any separable infinite dimensional
complex Banach space supports a mixing group $\{T_z\}_{z\in\C}$.

\subsection{Remarks on Theorem~\ref{zo}}

For completeness of the picture we include a generalization of
Lemma~\ref{re2} in the case $\K=\C$.

\begin{lemma} \label{re3} Let $n\in\N$ and $T$ be a continuous linear
operator with dense range on a complex topological vector space $X$
such that $T^n$ is cyclic. Let also $z=e^{2\pi i/n}$. Then the
operator $S=T\oplus zT\oplus z^2 T\oplus {\dots} \oplus z^{n-1}T$ is
cyclic. \end{lemma}

\begin{proof} Let $x$ be a cyclic vector for $T^n$. Then
$L=\{r(T^n)x:r\in\pp\}$ is dense in $X$. Since $T$ has dense range,
the spaces $T(L),\dots,T^{n-1}(L)$ are also dense in $X$. It
suffices to verify that $u=(x,\dots,x)\in X^n$ is a cyclic vector
for $S$. Let $M=C(S,u)$, $0\leq k\leq n-1$ and $r\in \pp$. Then
$$
S^kr(S^n)u=(T^kr(T^n)x,z^kT^kr(T^n)x,\dots,z^{k(n-1)}T^kr(T^n)x)\in
M.
$$
Thus, $M$ contains the vectors of the shape
$(a,z^ka,\dots,z^{k(n-1)}a)$ for $a\in T^k(L)$ and $0\leq k\leq
n-1$. Since $M$ is closed and $T^k(L)$ is dense in $X$, we see that
$$
M\supseteq N_k=\{(a,z^ka,\dots,z^{k(j-1)}a):a\in X\}\quad\text{for
$0\leq k\leq n-1$.}
$$
Finally, the matrix $\{z^{kl}\}_{k,l=0}^{n-1}$ is invertible since
its determinant is a Van der Monde one. Invertibility of the latter
matrix implies that the union of $N_k$ for $0\leq k\leq n-1$ spans
$X^n$. Hence $M=X^n$ and therefore $u$ is a cyclic vector for $S$.
\end{proof}

Lemma~\ref{re3} shows that under the condition
$z_1^k={\dots}=z_n^k=1$, the requirements on the operator $T$ in
Theorem~\ref{zo} can be weakened. Namely, instead of supercyclicity
of $T$ it is enough to require cyclicity of $T^k$ and the density of
the range of $T$. For general $z_j$ this is however not true. For
instance, the Volterra operator $V\in L(L^2[0,1])$, $Vf(t)=\int_0^t
f(s)\,ds$ has dense range, all powers of $V$ are cyclic, while
$V\oplus 2V$ is non-cyclic \cite{msh}. The latter example also
provides means for an elementary proof of the fact that the Volterra
operator $V$ is not weakly supercyclic (= not supercyclic on
$L_2[0,1]$ with weak topology), which also follows from a rather
involved general theorem in \cite{msh}, providing a sufficient
condition for a bounded linear operator on a Banach space to be not
weakly supercyclic.

\begin{corollary}\label{wsc} The Volterra operator $V$ acting on
$L^2[0,1]$ is not weakly supercyclic.
\end{corollary}

\begin{proof} Consider the operator $J:L_2[0,1]\to L_2[0,1]$,
$Jf(x)=f\left(\frac{1-x}2\right)$. It is easy to see that
$2JV=V^*J$. We show that $V\oplus 2V$ is not cyclic. Indeed, assume
that $(f,g)$ is a cyclic vector for $V\oplus 2V$. Then $f\neq 0$ and
$J^*f\neq 0$ since $J^*$ is injective. Using the equation
$2JV=V^*J$, it is easy to verify that the orbit $\{(V\oplus
2V)^n(f,g):n\in\Z_+\}$ is lying in the kernel of the continuous
functional $\Phi(u,v)=\langle u,Jg\rangle-\langle v,J^*f\rangle$ on
$H=L^2[0,1]\times L^2[0,1]$, where $\langle x,y\rangle=\int_0^1
x(t)y(t)\,dt$. Since $J^*f\neq 0$, we have $\Phi\neq 0$, which
contradicts cyclicity of $(f,g)$ for $V\oplus 2V$. Since cyclicity
of an operator on a Banach space is equivalent to cyclicity with
respect to the weak topology, $V\oplus 2V$ is non-cyclic on $H$ with
weak topology. By Theorem~\ref{zo}, $V$  is not weakly supercyclic.
\end{proof}

It seems to be an appropriate place to reproduce the following
question by Sophie Grivaux.

\begin{question} \label{q4} Let $T$ be a continuous linear operator
on a Banach space $X$ such that $T\oplus T$ is cyclic. Does it
follow that $T^2$ is cyclic?
\end{question}

\subsection{Cyclicity of direct sums of operators satisfying
the Supercyclicity Criterion}

The following important sufficient condition of supercyclicity is
provided in \cite{msa} for Banach space operators. In the general
setting it is a corollary of the following result of B\'es and Peris
\cite[Theorem~2.3 and Remark~2.6]{bp}. Recall that an infinite
family $\cal F$ of continuous maps from a topological space $X$ to a
topological space $Y$ is called {\it hereditarily universal} if each
infinite subfamily of $\cal F$ is universal.

\begin{thmbp} \it Let $\{T_n\}_{n\in\Z_+}$ be a sequence of
pairwise commuting continuous linear operators with dense range on a
separable Baire metrizable topological vector space $X$. Then the
following conditions are equivalent:

\smallskip

\noindent {\rm(p1)} \ The family $\{T_n\oplus T_n:n\in\Z_+\}$ is
universal;

\smallskip

\noindent {\rm(p2)} \ There is an infinite set $A\subseteq \Z_+$
such that $\{T_n:n\in A\}$ is hereditarily universal;

\smallskip

\noindent {\rm(p3)} \ There exist a strictly increasing sequence
$\{n_k\}$ of non-negative integers, dense subsets $E$ and $F$ of $X$
and maps $S_k:F\to X$ for $k\in\Z_+$ such that $T_{n_k}x\to 0$,
$S_ky\to 0$ and $T_{n_k}S_ky\to y$ as $k\to\infty$ for any $x\in E$
and $y\in F$.
\end{thmbp}

We say that a continuous linear operator $T$ on a separable Baire
metrizable topological vector space $X$ satisfies the {\it
Supercyclicity Criterion} \cite{msa} if there exist a strictly
increasing sequence $\{n_k\}_{k\in\Z_+}$ of positive integers, a
sequence $\{s_k\}_{k\in\Z_+}$ of positive numbers, dense subsets $E$
and $F$ of $X$ and maps $S_k:F\to X$ for $k\in\Z_+$ such that
$T^{n_k}S_ky\to y$, $s_kT^{n_k}x\to 0$ and $s_k^{-1}S_ky\to 0$ as
$k\to \infty$ for any $x\in E$ and $y\in F$. From Theorem~BP it
immediately follows that any operator $T$ satisfying the
Supercyclicity Criterion is supercyclic. Moreover it follows that
$T$ satisfies the Supercyclicity Criterion if and only if $T\oplus
T$ is supercyclic.

\begin{lemma}\label{sc1} Let $T_j$ for $1\leq j\leq k$ be continuous
linear operators on separable Baire metrizable topological vector
spaces $X_j$ all satisfying the Supercyclicity Criterion with the
same sequence $\{n_l\}_{l\in\Z_+}$. Then there exists an infinite
subset $A$ of $\N$ and sequences $\{s_{j,l}\}_{l\in A,\ 1\leq j\leq
n}$ of positive numbers such that the family ${\cal
F}=\{s_{1,l}T_1^{l}\oplus{\dots}\oplus s_{k,l}T_k^{l}:l\in A\}$ is
hereditarily universal.
\end{lemma}

\begin{proof} Let $\{n_l\}_{l\in\Z_+}$ be a strictly increasing sequence
of positive integers such that each $T_j$ satisfies the
Supercyclicity Criterion with this sequence. Then, for each
$j\in\{1,\dots,k\}$, we can pick a sequence
$\{s_{j,n_l}\}_{l\in\Z_+}$ of positive numbers, dense subsets $E_j$
and $F_j$ of $X_j$ and maps $S_{j,l}:F_j\to X_j$ for $l\in\Z_+$ such
that $T_j^{n_l}S_{j,l}y\to y$, $s_{j,n_l}T_j^{n_l}x\to 0$ and
$s_{j,n_l}^{-1}S_{j,l}y\to 0$ as $l\to \infty$ for any $x\in E_j$
and $y\in F_j$. Then $F=F_1\times{\dots}\times F_k$ and
$E=E_1\times{\dots}\times E_k$ are dense in $X=X_1\times
{\dots}\times X_k$. Let
$$
S_l=s_{1,n_l}^{-1}S_{1,n_l}\oplus{\dots}\oplus
s_{k,n_l}^{-1}S_{k,n_l}:F\to X\ \ \text{and}\ \
T_l=s_{1,n_l}T_1^{n_l}\oplus {\dots}\oplus s_{k,n_l}T_k^{n_l}\in
L(X).
$$
From the above properties of $s_{j,n_l}$ and $S_{j,l}$ it follows
that $T_lS_{l}y\to y$, $T_lx\to 0$ and $S_ly\to 0$ as $l\to \infty$
for any $x\in E$ and $y\in F$. It is also easy to see that $T_l$
have dense range and commute with each other. By Theorem~BP, there
is an infinite subset $B$ of $\Z_+$ such that the family $\{T_l:l\in
B\}$ is hereditarily universal. Clearly this family has the required
shape.
\end{proof}

The following result shows that if we restrict ourselves to
operators, satisfying the Supercyclicity Criterion, then cyclicity
of finite direct sums is satisfied not only for scalar multiples of
the same operator.

\begin{theorem}\label{f1} Let $T_j$ for $1\leq j\leq k$ be continuous
linear operators on separable Baire metrizable topological vector
spaces $X_j$ all satisfying the Supercyclicity Criterion with the
same sequence $\{n_l\}_{l\in\Z_+}$. Then the direct sum $T_1\oplus
{\dots}\oplus T_k$ is cyclic.
\end{theorem}

\begin{proof} By Lemma~\ref{sc1}, there
exists an infinite subset $A$ of $\Z_+$ and positive numbers
$\{s_{j,l}\}_{l\in A,\ 1\leq j\leq k}$ such that the family ${\cal
F}=\{S_l=s_{1,l}T_1^{l}\oplus{\dots}\oplus s_{k,l}T_k^{l}:l\in A\}$
is hereditarily universal. Denote $s_l=(s_{1,l},\dots,s_{k,l})\in
\R_+^k$. Replacing $A$ by a smaller infinite subset of $\Z_+$, if
necessary, may assume that $s_l/\|s_l\|_\infty\to s\in\R_+^k$ as
$l\to\infty$, $l\in A$. Clearly $\|s\|_\infty=1$. Without loss of
generality, we can also assume that $1=s_1\geq s_2\geq{\dots}\geq
s_k$. Let $m\in\N$ be such that $1\leq m\leq k$, $s_m\neq 0$ and
$s_j=0$ if $m<j\leq k$. Let $x=(x_1,\dots,x_k)$ be a universal
vector for the family $\cal F$. It suffices to show that $x$ is a
cyclic vector for $S=T_1\oplus{\dots}\oplus T_k$. We shall use
induction with respect to $k$. For $k=1$, the statement is trivial.
Assume that $k\geq 2$ and it is true for the direct sums of less
than $k$ operators. Denote $X=X_1\times{\dots}\times X_k$. First, we
shall show that $N=\{u\in X:u_j=0\ \text{for}\ j>m\}$ is contained
in the cyclic subspace $C(S,x)$. Let $w\in N$. Since $x$ is
universal for $\cal F$, there exists a strictly increasing sequence
$\{n_l\}_{l\in\Z_+}$ of elements of $A$ such that
$s_{j,n_l}T_j^{n_l}x_j\to w_j/s_j$ if $1\leq j\leq m$ and
$s_{j,n_l}T_j^{n_l}x_j\to 0$ if $j>m$. Consider the polynomials
$p_l(z)=\|s_{n_l}\|_{\infty}z^{n_l}$. Since $s_k/\|s_k\|_\infty\to
s\in\R_+^n$ as $k\to\infty$, $k\in A$, we obtain $p_l(S)x\to
(w_1,\dots,w_m,0,\dots,0)=w$ as $l\to\infty$, $l\in A$. Hence $w\in
C(S,x)$ and therefore $N\subseteq C(S,x)$. Clearly $S(N)\subseteq
N$. On the other hand $X/N$ is naturally isomorphic to
$X_{m+1}\times{\dots}\times X_n$ and the quotient operator
$\widetilde S(u+N)=Su+N$, acting on $X/N$ is naturally similar to
$T_{m+1}\oplus{\dots}\oplus T_n$. From the induction hypothesis it
follows that $x+N$ is a cyclic vector for $\widetilde S$. According
to Lemma~\ref{elem}, $x$ is a cyclic vector for $S$.
\end{proof}

\subsection{Strongly $n$-supercyclic operators}

Recently Feldman \cite{fe1} has introduced the notion of an
$n$-supercyclic operator for $n\in\N\cup\{\infty\}$. A continuous
linear operator $T$ on a topological vector space $X$ is called {\it
$n$-supercyclic} for $n\in\N$ if there exists an $n$-dimensional
linear subspace $L$ of $X$ such that its orbit $\{T^nx:n\in\Z_+,\
x\in L\}$ is dense in $X$. Such a space $L$ is called an $n$-{\it
supercyclic subspace} for $T$. Clearly, $1$-supercyclicity coincides
with the usual supercyclicity. In \cite{fe1}, for any $n\in\N$,
$n\geq 2$, a bounded linear operator $T$ on $\ell_2$ is constructed,
which is $n$-supercyclic and not $(n-1)$-supercyclic. In \cite{fe2},
the question is raised whether powers of $n$-supercyclic operators
are $n$-supercyclic. The question remains open. It is worth
mentioning that the answer becomes affirmative if we replace
$n$-supercyclicity by a slightly stronger property.

Namely, if $X$ is a topological vector space of dimension $\geq
n\in\N$, then the set $\PP_n X$ of all $n$-dimensional linear
subspaces of $X$ can be endowed with the natural topology. That is
we consider the (open) subset $X_n$ of all linearly independent
$n$-tuples $x=(x_1,\dots,x_n)\in X^n$ with the topology induced from
$X^n$ and declare the map $\pi_n:X_n\to \PP_n X$,
$\pi_n(x)=\spann\{x_1,\dots,x_n\}$ continuous and open. Observe that
$\PP_n \K^m$  for $m\geq n$ is the classical Grassmanian manifold.
We say that $L\in \PP_n X$ is a {\it strongly $n$-supercyclic}
subspace for $T\in L(X)$ if each $T^k(L)$ is $n$-dimensional and the
sequence $\{T^k(L):k\in\Z_+\}$ is dense in $\PP_n X$. An operator is
{\it strongly $n$-supercyclic} if it has a strongly $n$-supercyclic
subspace. Clearly, strong $1$-supercyclicity is equivalent to
supercyclicity and strong $n$-supercyclicity implies
$n$-supercyclicity. The advantage of strong $n$-supercyclicity lies
in the fact that it reduces to universality of the self-map of
$\PP_n X$ induced by $T$. Thus, using connectedness of $\PP_n X$ we
can prove that the powers of a strongly $n$-supercyclic operators
are strongly $n$-supercyclic by means of applying
Proposition~\ref{l1} in pretty much the same way as we did it for
supercyclic operators. It is worth noting that $n$-supercyclic
operators, constructed by Feldman in \cite{fe1}, are in fact
strongly $n$-supercyclic. This leads to the following question.

\begin{question}\label{q2} Are $n$-supercyclicity and strong
$n$-supercyclicity equivalent?
\end{question}

We would also like to reproduce the following interesting question
raised in \cite{fe1,fe2}.

\begin{question}\label{q3} Let $n\in \N$ and $T$ be an $n$-supercyclic operator
on a complex topological vector space $X$ such that
$\sigma_p(T^*)=\varnothing$. Is it true that $T$ is cyclic?
\end{question}

\subsection{Supercyclic operators with non-empty point spectrum}

Our final remark concerns the following claim made in \cite{leon}:

\begin{itemize}
\item[(L)] any supercyclic operator $T$ on a complex Banach space $X$
satisfying $\sigma_p(T^*)=\{z\}$ with $z\in\C\setminus\{0\}$ is
similar to the operator of the shape $z(S\oplus I_{\C})$, where $S$
is a hypercyclic operator and $I_{\C}$ is the identity operator on
the one-dimensional space $\C$.
\end{itemize}

If (L) was true, we could have significantly simplified the proof of
the characterization of $\R_+$-supercyclicity in Section~\ref{rsup}.
Unfortunately, this statement is false. It will be shown along with
a characterization of supercyclic operators with non-empty
$\sigma_p(T^*)$.

If $T$ is a supercyclic operator on a complex topological vector
space, then by Theorem~W, either $\sigma_p(T^*)=\varnothing$ or
$\sigma_p(T^*)=\{z\}$ with $z\in\C\setminus\{0\}$. In the latter
case, there is a non-zero $f\in X^*$ such that $T^* f=zf$. It is
clear that $Y=\ker f$ if a closed invariant subspace for $T$ of
codimension $1$. Pick any $e\in X$ such that $f(e)=1$. Then
$X=Y\oplus \langle e\rangle\simeq Y\times\C$. Since $T^* f=zf$, we
have $Te=z(e+u)$, where $u\in Y$. Thus, $T$ is naturally similar to
the operator $zS_u\in L(Y\times \C)$, where $S_u(y,t)=(Sy+tu,t)$ and
$S=z^{-1}T\bigr|_{Y}\in L(Y)$. Thus, any supercyclic operator is
similar to a  scalar multiple of an operator of the shape $S_u$. It
remains to figure out when an operator $S_u$ is supercyclic.

\begin{lemma}\label{supsup1} Let $Y$ be a topological vector space $S\in
L(Y)$, $u\in Y$ and
\begin{equation} \label{su}
S_u\in L(Y\times \K),\quad S_u(y,t)=(Sy+tu,t).
\end{equation}
Then $(x,s)$ is a supercyclic vector for $S_u$ if and only if $s\neq
0$ and $u-s^{-1}(I-S)x$ is a universal vector for the family
$\{p_n(S):n\in\N\}$, where
\begin{equation} \label{pn}
p_n(z)=\sum_{j=0}^{n-1} z^j.
\end{equation}
\end{lemma}

\begin{proof} If $s=0$, then $O(S_u,(x,s))$ is contained in
$Y\times\{0\}$ and $(x,s)$ can not even be a cyclic vector for
$S_u$. Assume now that $s\neq 0$. Then $(x,s)$ is a supercyclic
vector for $S_u$ if and only if $(y,1)$ is a supercyclic vector for
$S_u$, where $y=s^{-1}x$. Direct calculation shows that for any
$n\in\N$,
$$
S_u^n(y,1)=(S^ny+p_n(S)u,1)=(y+p_n(S)(u-(I-S)y),1).
$$
It immediately follows that $\{zS_u^n(y,1):n\in\Z_+,\ z\in\K\}$ is
dense in $Y\times \K$ if and only if $\{y+p_n(S)(u-(I-S)y):n\in\N\}$
is dense in $Y$. Since the translation by $y$ is a homeomorphism
from $Y$ to itself, the latter happens if and only if
$\{p_n(S)(u-(I-S)y):n\in\N\}$ is dense in $Y$. Thus, $(y,1)$ is a
supercyclic vector for $S_u$ if and only if $u-(I-S)y$ is a
universal vector for $\{p_n(S):n\in\N\}$. It remains to recall that
$y=s^{-1}x$.
\end{proof}

\begin{corollary}\label{supsup2} Let $Y$ be a topological vector space $S\in
L(Y)$, $u\in Y$ and $S_u\in L(Y\times \K)$  be defined by
$(\ref{su})$. Then $S_u$ is supercyclic if and only if the coset
$u+(I-S)(Y)$ contains a universal vector for the family
$\{p_n(S):n\in\N\}$, where $p_n$ are polynomials defined in
$(\ref{pn})$.
\end{corollary}

It becomes a natural question to study universality of the family
$\{p_n(S):n\in\N\}$.

\begin{lemma}\label{supsup3} Let $Y$ be a topological vector space,
$S\in L(Y)$ and ${\cal F}=\{p_n(S):n\in\N\}$, where $p_n$ are
polynomials defined by $(\ref{pn})$. Then $\cal F$ is universal if
and only if $S$ is hypercyclic. Moreover,
$(I-S)(\uu(S))\subseteq\uu({\cal F})\subseteq \uu(S)$.
\end{lemma}

\begin{proof} If $S$ is hypercyclic, then by Theorem~W, $I-S$ has
dense range. The same holds true if $\cal F$ is universal. Indeed,
assume that there is $x\in\uu({\cal F})$  and $I-S$ has non-dense
range. Then the closure $Z$ in $Y$ of $I-S(Y)$ is invariant for $S$
and $Z\neq Y$. Consider the operator $\widetilde S:Y/Z\to Y/Z$,
$\widetilde S(y+Z)=Sy+Z$. Clearly $x+Z$  is universal for ${\cal
A}=\{p_n(\widetilde S):n\in\N\}$. On the other hand, from the
definition of $Z$ it follows that $\widetilde S$ is the identity
operator on $Y/Z$ and therefore ${\cal A}=\{nI:n\in\N\}$. The latter
system is obviously non-universal. This contradiction shows that
$I-S$ has dense range if $\cal F$ is universal. Thus, we can assume
from the beginning that $(I-S)(Y)$ is dense in $Y$. It follows that
if $x\in\uu({\cal F})$, then $(I-S)x\in \uu({\cal F})$. Since
$p_n(S)(I-S)x=x-S^nx$, universality of $(I-S)x$ for $\cal F$ implies
that $x$ is hypercyclic for $S$. Hence universality of $\cal F$
implies hypercyclicity of $S$ and $\uu({\cal F})\subseteq \uu(S)$.
Now if $x\in \uu(S)$, then the sequence $p_n(S)(I-S)x=x-S^nx$ is
dense in $X$ and therefore $(I-S)x\in \uu({\cal F})$. Thus
hypercyclicity of $S$ implies universality of $\cal F$ and
$(I-S)(\uu(S))\subseteq\uu({\cal F})$. \end{proof}

\begin{lemma}\label{supsup4} Let $Y$ be a topological vector space, $S\in
L(Y)$, $u\in Y$ and $S_u\in L(Y\times \K)$  be defined by
$(\ref{su})$. Assume also that $1\notin\sigma_p(S^*)$. Then $S_u$ is
similar to $S_0=S\oplus I_\K$ if and only if $u\in (I-S)(Y)$.
\end{lemma}

\begin{proof} First, assume that $u\in (I-S)(Y)$. Then $u=v-Sv$ for
some $v\in Y$. Consider the operator $\Lambda\in L(Y\times \K)$,
$\Lambda(x,s)=(x+sv,s)$. Clearly $\Lambda$ is invertible and
$\Lambda^{-1}\in L(Y\times \K)$, $\Lambda^{-1}(x,s)=(x-sv,s)$. It is
easy to see that $\Lambda^{-1}S_u\Lambda=S_0$. Thus, $S_u$  is
similar to $S_0$.

Assume now that $S_u$ is similar to $S_0$. That is, there exists
$\Lambda\in L(Y\times\K)$ such that $\Lambda$ is invertible,
$\Lambda^{-1}$ is continuous and $\Lambda^{-1}S_u\Lambda=S_0$. Since
$\sigma_p(S^*)=\varnothing$, $\ker(S_0^*-I)$ is the one-dimensional
space spanned by the functional $f_0\in (Y\times \K)^*$,
$f_0(y,t)=t$. Since $(S_u^*-I)f_0=0$, we see that $\ker
f_0=Y\times\{0\}$ must be $\Lambda$-invariant. Since
$S_0(0,1)=(0,1)$ and $(0,1)\notin Y\times\{0\}$, $\Lambda_u$ must
have an eigenvector $(x,t)\notin Y\times\{0\}$ corresponding to the
eigenvalue $1$. Since $t\neq 0$, we can, without loss of generality,
assume that $t=1$. Thus, $S_u(x,1)=(x,1)$. It follows that $Sx+u=x$.
That is, $u=(I-S)x\in (I-S)(X)$.
\end{proof}

\begin{corollary}\label{supsup5} Let $Y$ be a topological vector space $S\in
L(Y)$, $u\in Y$ and $S_u\in L(Y\times \K)$  be defined by
$(\ref{su})$. If $u$ is universal for $\{p_n(S):n\in\N\}$, where
$p_n$ are polynomials defined by $(\ref{pn})$ and $u\notin
(I-S)(Y)$, then $S_u$ is supercyclic and not similar to $S_0=S\oplus
I_\K$.
\end{corollary}

\begin{proof} Supercyclicity of $S_u$ follows from
Corollary~\ref{supsup2}. By Lemma~\ref{supsup3}, $S$ is hypercyclic
and therefore by Theorem~S, $\sigma_p(S^*)=\varnothing$. By
Lemma~\ref{supsup4}, $S_u$ is not similar to $S_0=S\oplus I_\K$.
\end{proof}

\begin{corollary}\label{supsup6} Let $Y$ be a Baire separable metrizable
topological vector space and $S\in L(Y)$ be a hypercyclic operator
such that $(I-S)(Y)\neq Y$. Then there is $u\in Y$ such that the
operator $S_u$ defined by $(\ref{su})$ is supercyclic and not
similar to $S_0=S\oplus I_\K$.
\end{corollary}

\begin{proof} Since $S$ is hypercyclic, Lemma~\ref{supsup3} implies
that the family ${\cal F}=\{p_n(S):n\in \N\}$ with $p_n$ being the
polynomials defined in (\ref{pn}) has dense set of universal
elements. Since the set of universal elements for any family of maps
taking values in a second countable topological space is a
$G_\delta$-set \cite{ge1}, $\uu({\cal F})$ is a dense
$G_\delta$-subset of the Baire space $Y$. Hence $\uu({\cal F})$ is
not contained in $(I-S)(Y)$. Indeed, assume that $\uu({\cal
F})\subseteq (I-S)(Y)$. Since $(I-S)(Y)\neq Y$, there is $w\in Y$
such that $w\notin (I-S)(Y)$. Since $(I-S)(Y)$ is a linear subspace
of $Y$, $(I-S)(Y)\cap (w+(I-S)(Y))=\varnothing$. Hence $\uu({\cal
F})\cap (w+\uu({\cal F}))=\varnothing$. Thus $\uu({\cal F})$ and
$w+\uu({\cal F})$ are disjoint dense $G_\delta$-subsets in $Y$. The
existence of such subsets is impossible since $Y$ is Baire. Thus, we
can choose $u\in \uu({\cal F})\setminus (I-S)(Y)$. According to
Corollary~\ref{supsup5}, $S_u$ is supercyclic and not similar to
$S_0=S\oplus I_\K$. \end{proof}

To illustrate the above corollary, we consider the weighted backward
shifts on $\ell_2$. If $\{w_n\}_{n\in\N}$ is a bounded sequence of
positive numbers, then the operator $T_w:\ell_2\to\ell_2$ acting on
the canonical basis as $T_we_0=0$, $T_we_n=w_ne_{n-1}$ for $n>0$ is
called an {\it backward weighted shift}. Clearly $T_w$ is compact if
and only if $w_n\to 0$ as $n\to\infty$. According to Salas
\cite{salas}, any operator $S=I+T_w$ is hypercyclic. On the other
hand, $I-S=-T_w$ is not onto whenever the sequence  $w$ is not
bounded from below by a positive constant. In particular, for any
compact backward weighted shift $T$, $S=I+T$ is hypercyclic and
$I-S$ is not onto. Thus, by Corollary~\ref{supsup6}, we can choose
$u\in \ell_2$ such that $S_u$ is supercyclic and not similar to
$S\oplus I_\K$. Thus, the statement (L) is indeed false.

From Lemmas~\ref{supsup1} and~\ref{supsup3} it follows that
supercyclicity of $S_u$ implies hypercyclicity of $S$. Moreover,
supercyclicity of $S_0$ is equivalent to hypercyclicity of $S$. It
is natural therefore to consider the question whether hypercyclicity
of $S$ implies supercyclicity of $S_u$ for any vector $u$. The
following example provides a negative answer to this question even
in the friendly situation when $Y$ is a Hilbert space.

\begin{proposition} There exists a hypercyclic operator $S\in
L(\ell_2)$ and $u\in\ell_2$ such that the operator $S_u\in
L(\ell_2\times \K)$ defined by $(\ref{su})$ is not supercyclic.
\end{proposition}

\begin{proof} Consider the backward weighted shift $T\in L(\ell_2)$
with the weight sequence $w_n=e^{-2n}$, $n\in\N$ and let $S=I+T$.
Let also $u\in \ell_2$, $u_n=(n+1)^{-1}$ for $n\in\Z_+$. According
the above cited theorem of Salas, $S$ is hypercyclic. It remains to
demonstrate that $S_u$ is not supercyclic.

Assume that $S_u$ is supercyclic. By Corollary~\ref{supsup2}, the
set $u+T(\ell_2)$ contains a universal vector for the family ${\cal
F}=\{p_n(S):n\in \N\}$ with $p_n$ being the polynomials defined in
(\ref{pn}). By Lemma~\ref{supsup3}, $\uu({\cal F})\subseteq\uu(S)$
and therefore $(u+T(\ell_2))\cap \uu(S)\neq\varnothing$. That is,
there exists $x\in\ell_2$ such that $u+Tx$ is a hypercyclic vector
for $S$. We are going to obtain a contradiction by showing that
$\|S^n(u+Tx)\|\to\infty$ as $n\to\infty$.

Taking into account that $S^n=(I+T)^n=\sum\limits_{k=0}^n \bin nk
T^k$, we see that for each $y\in\ell_2$,
\begin{equation}\label{tn}
(S^ny)_j=\sum_{k=0}^n \bin nk y_{j+k}e^{-k(k+1)}\ \ \text{for any
$j,n\in\Z_+$}.
\end{equation}
Applying this formula to $y=u$, we obtain
\begin{equation}\label{tn1}
\|S^nu\|\geq |(S^nu)_0|=A_n,\ \ \text{where}\ \ A_n= \sum_{k=0}^n
\bin nk (k+1)^{-1}e^{-k(k+1)}\ \ \text{for any $n\in\Z_+$}.
\end{equation}
Now since $x\in\ell_2$, $c=\sup\limits_{j\in\Z_+}|x_j|<\infty$.
Hence $|(Tx)_j|\leq ce^{-2j-2}$. Substituting these inequalities
into (\ref{tn}), we see that
\begin{align}\label{tn2}
&|(S^nTx)_j|\leq c \sum_{k=0}^n \bin nk e^{-2j-2k-2}e^{-k(k+1)}=
ce^{-2j} B_n\ \text{for any $n,j\in\Z_+$},
\\
&\quad\text{where}\ \ B_n= \sum_{k=0}^n \bin nk e^{-(k+1)(k+2)}.
\notag
\end{align}
Summing up the inequalities in (\ref{tn2}), we obtain
\begin{equation}\label{tn3}
\|S^nTx\|\leq \sum_{j=0}^\infty |(S^nTx)_j|\leq
cB_n\sum_{j=0}^\infty e^{-2j}<2cB_n\ \ \text{for any $n\in\Z_+$.}
\end{equation}

In order to demonstrate that $\|S^n(u+Tx)\|\to\infty$ as
$n\to\infty$, it is enough to show that $A_n\to \infty$ and
$B_n=o(A_n)$ as $n\to\infty$. Indeed, then from (\ref{tn1}) and
(\ref{tn3}) it immediately follows that $\|S^n(u+Tx)\|\to\infty$.
Applying the Stirling formula to estimate $\bin nk
(k+1)^{-1}e^{-k(k+1)}$, we see that there exist positive constants
$\alpha$ and $\beta$ such that whenever $1\leq k\leq n^{1/2}$,
\begin{equation}\label{676}
\alpha k^{-3/2}(nk^{-1}e^{-k})^k\leq \bin nk
(k+1)^{-1}e^{-k(k+1)}\leq \beta k^{-3/2}(nk^{-1}e^{-k})^k.
\end{equation}
From (\ref{676}) it follows that if $\{k_n\}$ is a sequence of
positive integers such that $2k_n-\ln n=O(1)$, then for any $a<1/4$,
$$
A_n\geq \bin n{k_n} (k_n+1)^{-1}e^{-k_n(k_n+1)}\geq e^{a(\ln n)^2}\
\ \text{for all sufficiently large $n$.}
$$
Hence $A_n\to\infty$. Using (\ref{676}), we immediately see that
$$
A'_n= \sum_{0\leq k<(\ln n)/4} \bin nk
(k+1)^{-1}e^{-k(k+1)}=o(e^{b(\ln n)^2}) \ \ \text{for any $b>3/16$.}
$$
From the last two displays we have $A'_n=o(A_n)$. Then
$$
B'_n=\sum_{0\leq k<(\ln n)/4} \bin nk e^{-(k+1)(k+2)}\leq
A_n'=o(A_n).
$$
On the other hand, for $(\ln n)/4\leq k\leq n$, we have
$(k+1)e^{-2k-2}<4n^{-2}\ln n$ and therefore
$$
\bin nk e^{-(k+1)(k+2)}\leq  \bin nk (k+1)^{-1}e^{-k(k+1)}4n^{-2}\ln
n.
$$
Hence
$$
(B_n-B'_n)\leq 4n^{-1}\ln n (A_n-A'_n)\leq 4n^{-2}\ln n A_n=o(A_n).
$$
Thus, $B_n=B_n'+(B_n-B'_n)=o(A_n)$, which completes the proof.
\end{proof}

\section{Appendix: Proof of Lemmas~\ref{conn1}--\ref{conn4}} \small

Recall that a subset $U$ of a topological vector space $X$ is called
{\it balanced} if $tx\in U$ for any $x\in U$ and $t\in \K$ with
$|t|\leq 1$. It is well-known \cite{shifer} that any topological
vector space has a base of neighborhoods of zero consisting of
balanced sets.

\begin{proof}[Proof of Lemma~\ref{conn1}]
Let $x,y\in X$. Then $f_{x,y}:[0,1]\to X$, $f_{x,y}(t)=(1-t)x+ty$ is
continuous, $f_{x,y}(0)=x$ and $f_{x,y}(1)=y$. If ${\cal W}$ is a
base of neighborhoods of zero in $X$ consisting of balanced sets,
then $\{x+W:x\in X,\ W\in{\cal W}\}$ is a base of topology of $X$
consisting of path connected sets. Indeed, for any $x\in X$,
$W\in{\cal W}$ and $w\in W$ the continuous path $f_{x,x+w}$ connects
$x$ and $x+w$ and never leaves $x+W$. Hence $X$ has a base of
topology consisting of path connected sets and therefore $X$ is
locally path connected. Finally, let $f:\T\to X$ be continuous. Then
$h:[0,1]\times \T\to X$, $h(t,s)=tf(s)$ is a contraction of $f$.
Thus, $X$ is simply connected.
\end{proof}

We shall use the following well-known properties of connectedness.

\begin{itemize}
\item[(c1)] A path connected topological space $X$ is simply connected
if and only if the space $C(\T,X)$ of continuous maps from $\T$ to
$X$ with the compact-open topology is path connected;
\item[(c2)] If $X$ is a locally connected (has a base of topology consisting
of connected sets), then any connected component of $X$ is closed
and open.
\end{itemize}

\begin{lemma}\label{dense1} Let $X$ be a topological vector space
and $A$ be the subset of $C(\T,X\setminus \{0\})$ consisting of
continuous functions $f:\T\to X\setminus\{0\}$ with finite
dimensional $\spann(f(\T))$. Then $A$ is dense in $C(\T,X\setminus
\{0\})$.
\end{lemma}

\begin{proof} Let $f\in C(\T,X\setminus \{0\})$. For any $n\in\N$
consider the function $f_n:\T\to X$ such that $f_n(e^{2\pi
it})=(k+1-nt)f(e^{2\pi ik/n})+(nt-k)f(e^{2\pi i(k+1)/n})$ for
$t\in[k/n,k+1/n]$, $0\leq k\leq n$. It is straightforward to see
that $f_n$ are well defined, continuous and the sequence $f_n$ is
convergent to $f$ in $C(\T,X)$. Taking into account that $f_n$ does
not take value zero for sufficiently large $n$ and $\spann(f_n(\T))$
is at most $n$-dimensional, we obtain density of $A$ in
$C(\T,X\setminus \{0\})$.
\end{proof}

\begin{proof}[Proof of Lemma~\ref{conn2}]
Let ${\cal W}$ be a base of neighborhoods of zero in $X$ consisting
of balanced sets. Consider the family ${\cal B}=\{x+W:x\in X,\
W\in{\cal W},\ 0\notin x+W\}$. Clearly $\cal B$ is a base of
topology of $X\setminus\{0\}$. As in the proof of Lemma~\ref{conn1},
we see that $\cal B$ consists of path connected sets and therefore
$X\setminus\{0\}$ is locally path connected. Next, for any $x,u\in
X\setminus\{0\}$, we can pick $y\in X$ such that $0\notin
[x,y]\cup[y,u]$.  Then $g_{x,y,u}:[0,1]\to X\setminus\{0\}$,
$$
g_{x,y,u}(t)=\left\{\begin{array}{ll}(1-2t)x+2ty&\text{if}\ \ 0\leq
t\leq 1/2,\\ (2-2t)y+(2t-1)u&\text{if}\ \ 1/2<t\leq
1\end{array}\right.
$$
is continuous, $g(0)=x$ and $g(1)=u$. Thus, $X\setminus\{0\}$ is
path connected.

Next, let $f\in C(\T,X\setminus\{0\})$. Pick $W\in{\cal W}$ such
that $0\notin f(\T)+W$. Then the set $\Omega_{f,W}=\{g\in
C(\T,X\setminus\{0\}):(g-f)(\T)\subseteq W\}$ is a neighborhood of
$f$ in $C(\T,X\setminus\{0\})$. Moreover, $\Omega_{f,W}$ is path
connected. Indeed, for any $g\in \Omega_{f,W}$, the continuous path
$F:[0,1]\to C(\T,x+W)$, $F(t)(s)=(1-t)g(s)+tf(s)$ connects $g$ and
$f$ and never leaves $\Omega_{f,W}$. Since the family
$\{\Omega_{f,W}:0\notin f(\T)+W\}$ is a base of topology of
$C(\T,X\setminus\{0\})$, we see that $C(\T,X\setminus\{0\})$ is
locally path connected. Assume that $X\setminus\{0\}$ is not simply
connected. According to (c1), $C(\T,X\setminus\{0\})$ is not path
connected. Thus, there is $f\in C(\T,X\setminus\{0\})$ that can not
be connected with a constant map by a continuous path. By (c2) the
connected component $\Omega$ of $f$ in $C(\T,X\setminus\{0\})$ is
open. According to Lemma~\ref{dense1}, the family $A$ of $g\in
C(\T,X\setminus\{0\})$ with finite dimensional $\spann(g(\T))$ is
dense in $C(\T,X\setminus\{0\})$. Hence we can pick $g\in \Omega$
and an $\R$-linear subspace $L$ of $X$ such that $3\leq \dim
L<\infty$ and $g(\T)\subseteq L$. On the other hand, it is
well-known \cite{at} that $\R^n\setminus\{0\}$ is simply connected
for $n\geq 3$. Thus, we can connect $g$ with a constant map by a
continuous path $x_0$ lying within $C(\T,L\setminus\{0\})\subseteq
C(\T,X\setminus\{0\})$. We have obtained a contradiction.
\end{proof}

\begin{lemma} \label{exten1} Let $X$ be a topological vector space
and $f:[0,1]\to \PP X$ be continuous. Then there exists continuous
$g:[0,1]\to X\setminus\{0\}$ such that $f=\pi\circ g$. Similarly,
for any continuous $f_0:[0,1]\to \PP_+ X$, there exists continuous
$g_0:[0,1]\to X\setminus\{0\}$ such that $f_0=\pi_+\circ g_0$.
\end{lemma}

\begin{proof} Taking into account that $f([0,1])$ and $f_0([0,1])$
are metrizable and compact, we see that $\pi^{-1}(f([0,1]))$ and
$\pi_+^{-1}(f_0([0,1]))$ are complete metrizable subsets of $X$.
Consider multivalued functions $\widetilde f:[0,1]\to 2^X$ and
$\widetilde f_0:[0,1]\to 2^X$, $\widetilde f(t)$ being $f(t)$
considered as a subset (one-dimensional subspace) of $X$ and
$\widetilde f_0(t)$ being $f_0(t)$ considered as a subset (ray) of
$X$. It is easy to see that these multivalued maps satisfy all the
conditions of Theorem~1.2 \cite{selec} by Michael and therefore for
any $t_0\in [0,1]$ and any $x_0\in \widetilde f(t_0)$ (respectively,
$x_0\in \widetilde f_0(t_0)$) there exists a continuous map
$h_{t_0,x_0}:[0,1]\to X$ such that $h_{t_0,x_0}(t_0)=x_0$ and
$h_{t_0,x_0}(t)\in \widetilde f(t)$ (respectively,
$h_{t_0,x_0}(t)\in \widetilde f_0(t)$) for any $t\in [0,1]$. Now it
is a routine exercise to show that required functions $g$ and $g_0$
can be obtained as finite sums of the functions of the shape
$\phi\cdot h_{t_0,x_0}$ with continuous $\phi:[0,1]\to (0,\infty)$.
\end{proof}

\begin{lemma} \label{exten2} Let $X$ be a complex topological vector
space. Then for any continuous $f:\T\to \PP X$, there exists
continuous $g:\T\to X\setminus\{0\}$ such that $f=\pi\circ g$.
\end{lemma}

\begin{proof} By Lemma~\ref{exten1}, we can find continuous $h:[0,1]\to
X\setminus\{0\}$ such that $f(e^{2\pi it})=\pi(h(t))$ for any $t\in
[0,1]$. Since $h(0)$ and $h(1)$ are both non-zero elements of the
one-dimensional space $f(1)$, there is $z\in\C\setminus \{0\}$ such
that $h(1)=z^{-1}h(0)$. Let $r>0$ and $s\in\R$ be such that
$z=re^{is}$. Now it is easy to see that the function $g:\T\to
X\setminus\{0\}$ defined by the formula $g(e^{2\pi i
t})=(1+(r-1)t)e^{ist}h(t)$ for $0\leq t<1$ satisfies the required
conditions.
\end{proof}

\begin{lemma} \label{exten3} Let $X$ be a topological vector
space. Then for any continuous $f:\T\to \PP_+ X$, there exists
continuous $g:\T\to X\setminus\{0\}$ such that $f=\pi_+\circ g$.
\end{lemma}

\begin{proof} By Lemma~\ref{exten1}, there is a continuous $h:[0,1]\to
X\setminus\{0\}$ such that $f(e^{2\pi it})=\pi(h(t))$ for any $t\in
[0,1]$. Since $h(0)$ and $h(1)$ are both non-zero elements of the
ray $f(1)$, there is $r>0$ such that $h(1)=r^{-1}h(0)$. Clearly, the
function $g:\T\to X\setminus\{0\}$ defined as $g(e^{2\pi i
t})=(1+(r-1)t)h(t)$ for $0\leq t<1$ satisfies the required
conditions.
\end{proof}

\begin{proof}[Proof of Lemma~\ref{conn3}]
For every linearly independent $x,y\in X$, we consider the path
$h_{x,y}:[0,1]\to\PP X$, $h_{x,y}(t)=\langle (1-t)x+ty\rangle$.
Clearly each $h_{x,y}$ is continuous, $h_{x,y}(0)=\langle x\rangle$
and $h_{x,y}(1)=\langle y\rangle$. Thus, any two distinct points of
$\PP X$ can be connected by a continuous path and therefore $\PP X$
is path connected. If $W$ is a balanced neighborhood of zero and
$x\in X$ are such that $0\notin x+W$, then for any $y\in x+W$ with
$\langle y\rangle\neq\langle x\rangle$, the path $h_{x,y}$ connects
$\langle x\rangle$ and $\langle y\rangle$ and never leaves
$\pi(x+W)$. Thus, $\pi(x+W)$ is path connected. Since the family of
$\pi(x+W)$ forms a base of topology of $\PP X$, we see that $\PP X$
is locally path connected. Let now $f:\T\to \PP X$ be a continuous
map. If $X$ is finite dimensional, then $\PP X$ is homeomorphic to
$\PP \C^n$ for some $n\in \N$. The latter spaces are known to be
simply connected \cite{at}. If $X$ is infinite dimensional, we use
Lemma~\ref{exten2} to find  a continuous  map $g:\T\to X\setminus
\{0\}$ such that $f=\pi\circ g$. According to Lemma~\ref{conn2},
there is continuous $\phi:[0,1]\times \T\to X\setminus\{0\}$ and
$x_0\in X\setminus\{0\}$ such that $\phi(0,s)=g(s)$ and
$\phi(1,s)=x_0$ for any $s\in\T$. The map $\pi\circ\phi$ provides
the required contraction of the closed path $f$ in $\PP X$.
\end{proof}

\begin{proof}[Proof of Lemma~\ref{conn4}]
The proof of path connectedness, local path connectedness of $\PP_+
X$ is exactly the same as the above proof for $\PP X$. Simple
connectedness of $\PP_+ X$ in the case of finite dimensional $X$
follows from the fact that then $\PP_+ X$ is homeomorphic to the
$n-1$-dimensional sphere $S^{n-1}$ if the real dimension of $X$ is
$n$. Since $S^k$ is simply connected for $k\geq 2$ \cite{at}, $\PP_+
X$ is simply connected if $X$ has finite real dimension $\geq 3$. If
$X$ is infinite dimensional, the proof of simple connectedness of
$\PP_+ X$ is the same as for $\PP X$ for complex $X$ with the only
difference that we use Lemma~\ref{exten3} instead of
Lemma~\ref{exten2}.
\end{proof}

{\bf Remark.} \ The relative difficulty of some of the above lemmas
is due to the fact that we consider general, not necessarily locally
convex, topological vector spaces. The locally convex case is indeed
elementary because of the guaranteed rich supply of continuous
linear functionals.


\small\rm

\vskip1truecm

\scshape

\noindent Stanislav Shkarin

\noindent Queens's University Belfast

\noindent Department of Pure Mathematics

\noindent University road, Belfast, BT7 1NN, UK

\noindent E-mail address: \qquad {\tt s.shkarin@qub.ac.uk}

\end{document}